\documentclass[10pt, a4paper, reqno, oneside]{amsart}

\usepackage{amsmath, amsfonts, amsthm, mathtools, enumerate, multicol}

\usepackage[mathscr]{euscript}
\usepackage{mathbbol}
\usepackage[latin1]{inputenc}
\usepackage{graphicx}
\usepackage[all]{xy}
\usepackage{enumitem}
\usepackage{autobreak}
\setlist[itemize]{noitemsep, topsep=1pt, leftmargin=20pt}
\usepackage{fullpage}

\usepackage[hypertexnames=false,
backref=page,
    pdfpagemode=UseNone,
    breaklinks=true,
    extension=pdf,
    colorlinks=true,
    linkcolor=blue,
    citecolor=blue,
    urlcolor=blue,
]{hyperref}

\newcommand\bcdot{\ensuremath{
  \mathchoice
   {\mskip\thinmuskip\lower0.2ex\hbox{\scalebox{1.6}{$\cdot$}}\mskip\thinmuskip}}
   {\mskip\thinmuskip\lower0.2ex\hbox{\scalebox{1.6}{$\cdot$}}\mskip\thinmuskip}
   {\lower0.3ex\hbox{\scalebox{1.2}{$\cdot$}}}
   {\lower0.3ex\hbox{\scalebox{1.2}{$\cdot$}}}
}
\theoremstyle{plain}
\newtheorem{theo}{Theorem}[section]

\newtheorem{prop}[theo]{Proposition}

\newtheorem{cor}[theo]{Corollary}
\theoremstyle{definition}
\newtheorem{rem}[theo]{Remark}
\newtheorem{example}[theo]{Example}
\newtheorem{definition}[theo]{Definition}

\theoremstyle{plain}
\newtheorem{lemma}[theo]{Lemma}
\newtheorem{theorem}[theo]{Theorem}
\newtheorem{corollary}[theo]{Corollary}

\theoremstyle{definition}

\newtheorem{remark}[theo]{Remark}

\theoremstyle{plain}
\newtheorem{thmint}{Theorem}

\renewcommand{\=}{\coloneqq}

\renewcommand{\a}{\alpha}
\renewcommand{\b}{\beta}

\newcommand{\e}{\varepsilon}
\newcommand{\f}{\varphi}
\newcommand{\g}{\gamma}

\renewcommand{\l}{\lambda}
\newcommand{\w}{\omega}
\newcommand{\q}{\vartheta}

\renewcommand{\t}{\tau}

\newcommand{\D}{\Delta}

\newcommand{\G}{\Gamma}

\renewcommand{\S}{\Sigma}

\newcommand{\W}{\Omega}

\DeclareSymbolFontAlphabet{\mathbb}{AMSb}
\DeclareSymbolFontAlphabet{\mathbbl}{bbold}

\newcommand{\bC}{\mathbb{C}}

\newcommand{\bR}{\mathbb{R}}
\newcommand{\bZ}{\mathbb{Z}}

\newcommand{\bJ}{\mathbb{J}}

\newcommand{\bN}{\mathbb{N}}

\newcommand{\bP}{\mathbb{P}}

\newcommand{\bg}{\mathbbl{g}}

\newcommand{\fE}{\mathsf{E}}

\newcommand{\fG}{\mathsf{G}}
\newcommand{\fH}{\mathsf{H}}

\newcommand{\fK}{\mathsf{K}}
\newcommand{\fL}{\mathsf{L}}

\newcommand{\fN}{\mathsf{N}}

\newcommand{\fS}{\mathsf{S}}

\newcommand{\fU}{\mathsf{U}}

\newcommand{\fZ}{\mathsf{Z}}

\newcommand{\fSO}{\mathsf{SO}}

\newcommand{\fSU}{\mathsf{SU}}
\newcommand{\fSp}{\mathsf{Sp}}
\newcommand{\fSpin}{\mathsf{Spin}}

\newcommand{\ga}{\mathfrak{a}}

\renewcommand{\gg}{\mathfrak{g}}
\newcommand{\gh}{\mathfrak{h}}

\newcommand{\gk}{\mathfrak{k}}

\newcommand{\gm}{\mathfrak{m}}

\newcommand{\gp}{\mathfrak{p}}

\newcommand{\gz}{\mathfrak{z}}

\newcommand{\su}{\mathfrak{su}}

\newcommand\Sym{\mathrm{Sym}}

\newcommand{\cC}{\mathcal{C}}

\newcommand{\cO}{\mathcal{O}}

\newcommand{\cU}{\mathcal{U}}

\newcommand{\eL}{\EuScript{L}}

\newcommand{\eS}{\EuScript{S}}

\newcommand{\p}{\partial}
\DeclareMathOperator\Tr{Tr}
\DeclareMathOperator\Lie{Lie}
\DeclareMathOperator\End{End} 
\DeclareMathOperator\scal{scal} 
\DeclareMathOperator\Ad{Ad}
\DeclareMathOperator\ad{ad}
\DeclareMathOperator\Id{Id}
\DeclareMathOperator\diff{d\!}
\newcommand\Ric{\operatorname{Ric}}

\newcommand{\ol}{\overline}
\newcommand{\zero}{\operatorname{o}}

\newcommand{\n}{\nabla}

\makeatletter
\@namedef{subjclassname@2020}{\textup{2020} Mathematics Subject Classification}
\makeatother

\allowdisplaybreaks[3]

\title[On cohomogeneity one non-K\"ahler metrics]{On cohomogeneity one Hermitian non-K\"ahler metrics}

\author{Daniele Angella}
\address[Daniele Angella]{Dipartimento di Matematica e Informatica ``Ulisse Dini''\\
Universit\`a degli Studi di Firenze,
viale Morgagni 67/a,
	50134 Firenze, Italy}
\email{daniele.angella@unifi.it}
\email{daniele.angella@gmail.com}

\author{Francesco Pediconi}
\address[Francesco Pediconi]{Dipartimento di Matematica e Informatica ``Ulisse Dini''\\
	Universit\`a degli Studi di Firenze,
	viale Morgagni 67/a,
	50134 Firenze, Italy}
\email{francesco.pediconi@unifi.it}

\thanks{
The authors are supported by project PRIN2017 ``Real and Complex Manifolds: Topology, Geometry and holomorphic dynamics'' (code 2017JZ2SW5), and by GNSAGA of INdAM
}

\begin{document}
\begin{abstract}
We investigate the geometry of Hermitian manifolds endowed with a compact Lie group action by holomorphic isometries with principal orbits of codimension one. In particular, we focus on a special class of these manifolds constructed by following B\'erard-Bergery which includes, among the others, the holomorphic line bundles on $\bC\bP^{m-1}$, the linear Hopf manifolds and the Hirzebruch surfaces. We characterize their invariant special Hermitian metrics, such as balanced, K\"ahler-like, pluriclosed, locally conformally K\"ahler, Vaisman, Gauduchon. Furthermore, we construct new examples of cohomogeneity one Hermitian metrics solving the second-Chern-Einstein equation and the constant Chern-scalar curvature equation.
\end{abstract}

\maketitle

\section{Introduction}

One of the most useful ways to construct concrete examples of Einstein metrics is by considering Riemannian manifolds with a large symmetry group, for example, homogeneous spaces and manifolds of cohomogeneity one, see {\it e.g.} \cite{besse, wang-1, wang-2} and references therein. As another useful tool, the Calabi-Yau theorem assures the existence of Einstein metrics on compact complex K\"ahler manifolds with non-positive first Chern class \cite{aubin, yau}, and also the existence of K\"ahler-Einstein metrics on Fano manifolds has been recently understood.

The first non-homogeneous example of compact Riemannian Einstein manifold with positive scalar curvature has been provided by Page \cite{page, page-pope} on $\bC\bP^2 \# \ol{\bC\bP}{}^2$, and then generalized by B\'erard-Bergery in \cite{berard-bergery} as follows. Let $P$ be a compact K\"ahler-Einstein manifold with positive scalar curvature \cite{alekseevsky-perelomov, koiso-sakane-1, koiso-sakane-2}, for example, $P=\mathbb C \bP^1$ in case of the Page example. Let the first Chern class be $c_1(P)=p\alpha$, with $p>0$ integer and $\alpha\in H^2(P;\mathbb Z)$ indivisible. For $n>0$ integer, consider the line bundle $\mathbb C \to E_n\to P$ with $c_1(E_n)=n\alpha$, and the associated projective bundle $\mathbb C \bP^1 \to M_n\to P$. Then one can use the theory of Riemannian submersion \cite{oneill} to compute the Riemannian curvature of both $E_n$ and $M_n$. In particular, when $P=\fG/\fK$ is a Hermitian symmetric space, then $E_n$ and $M_n$ are cohomogeneity one, and by consequence their curvature is determined by the structure constants of the Lie algebra $\gg\=\Lie(\fG)$ and by the induced one-parameter family of scalar products \cite{berard-bergery}. The Einstein equation is then reduced to a system of second-order ODEs, in both the spaces $E_n$ and $M_n$, that can be integrated to get Einstein metrics which are also Hermitian, see \cite[Th\'eor\`eme 1.10 and Th\'eor\`eme 1.13]{berard-bergery}. \smallskip

In the complex setting, K\"ahler C-spaces ({\itshape i.e.} compact, simply-connected, homogeneous K\"ahler manifolds) are well understood \cite{wang, alekseevsky-perelomov, bordermann-forger-romer}: in particular, they always admit a unique K\"ahler-Einstein metric, up to scaling. On the other hand, cohomogeneity one K\"ahler metrics have been studied in \cite{sakane, koiso-sakane-2, dancer-wang, podesta-spiro, guan-chen, podesta-spiro-2, alekseevsky-zuddas-1, alekseevsky-zuddas-2}.

In the Hermitian non-K\"ahlerian context, since the Levi-Civita connection does not preserve the complex structure, one is led to introduce a more suitable notion of canonical metrics, for example, the {\em second-Chern-Einstein} condition.
Here, by second-Chern-Einstein metric on a complex manifold, we mean a Hermitian metric on the tangent bundle that is Hermite-Einstein with respect to itself \cite{gauduchon}, see also \cite{streets-tian, angella-calamai-spotti-2}.
Examples of compact second-Chern-Einstein manifolds include the homogeneous Hopf surface \cite{gauduchon, gauduchon-ivanov, liu-yang}, holomorphically-parallelizable manifolds \cite{boothby}, and the homogeneous non-K\"ahler C-spaces studied by Podest\`a \cite{podesta}; see also \cite{alekseevsky-podesta} for the almost-K\"ahler case.
In particular, the only compact complex non-K\"ahler surface admitting second-Chern-Einstein metrics is the Hopf surface \cite[Theorem 2]{gauduchon-ivanov}, see also \cite{gauduchon}.
Note that there still miss (if any) non-K\"ahler examples of second-Chern-Einstein metrics with negative Chern-scalar curvature on compact complex manifolds. Further problems, {\itshape e.g.} the {\em constant Chern-scalar curvature problem}, also known as Chern-Yamabe problem \cite{angella-calamai-spotti}, are still not completely understood. Non-homogeneous examples of Hermitian metrics of positive constant Chern-scalar curvature on Hirzebruch surfaces, using the Page and B\'erard-Bergery ansatz, have been constructed in \cite{koca-lejmi}. \smallskip

In this note, motivated by the above questions, we start to investigate the curvature and the properties of Hermitian non-K\"ahler manifolds with large isometry groups. In particular, we focus on the {\em B\'erard-Bergery} \cite{berard-bergery} {\em standard} \cite{podesta-spiro} {\em cohomogeneity one Hermitian manifolds} $(M_{(i,n)}(\fG,\fK),\bJ,\bg)$, with $i \in \{1,2,3,4\}$ and $n \in \bN$, that are defined as (the total spaces of) bundles over a simply connected, irreducible, compact Hermitian symmetric space $P=\fG/\fK$, see Definition \ref{def:bb}. We compute the Chern-Ricci and Chern-scalar curvatures, see Proposition \ref{prop:chern-ricci}.
This is aimed at getting the second-Chern-Einstein and the constant Chern-scalar curvature equations, that we investigate in the last section.
Here, we continue by describing the existence of {\em special} metrics. \smallskip

Our first result concerns the locally conformally K\"ahler condition. We recall that a Hermitian metric $\bg$ is called {\em locally conformally K\"ahler} if it admits a local conformal change to a K\"ahler metric \cite{dragomir-ornea}, which is equivalent to say that $\diff\omega=\frac{1}{m-1}\q\wedge\omega$ with $\diff\q=0$, where $\omega=\bg(\bJ \cdot, \cdot)$ is the fundamental $2$-form associated to the metric, $m$ denotes the complex dimension and $\q$ is the {\em Lee form} of $\bg$ \cite{gauduchon-2}, see Section \ref{sec:curvature-bb}.
In the particular case when $\q$ is also exact, say $\q=(m-1)\diff\phi$, then $\exp(-\phi)\omega$ is K\"ahler, and $\bg$ is called {\em globally conformally K\"ahler} (this includes the case $\q=0$, corresponding to $\bg$ itself being K\"ahler).
When $\q$ is not exact, we say that the metric is {\em strictly locally conformally K\"ahler}.
Moreover, a strictly locally conformally K\"ahler metric is called {\em Vaisman} if the Lee form $\q$ is parallel with respect to the Levi-Civita connection.

It is worth noticing that homogeneous, strictly locally conformally K\"ahler metrics on compact manifolds are Vaisman \cite{gauduchon-moroianu-ornea, hasegawa-kamishima}. On the other hand, in the cohomogeneity one case, we prove the following

\begin{thmint}[see Proposition \ref{prop:LCK}, Corollary \ref{globCK} and Proposition \ref{prop:vaisman}]\label{thm:a}
Let $(M_{(i,n)}(\fG,\fK), \bJ, \bg)$ be one of the B\'erard-Bergery standard cohomogeneity one Hermitian manifolds. Then, $\bg$ is locally conformally K\"ahler. Furthermore, $\bg$ is strictly locally conformally K\"ahler if and only if $(M_{(i,n)}(\fG,\fK), \bJ)$ is compact without singular orbits (case $i=3$) and, on such manifolds, $\bg$ is Vaisman if and only if it is homogeneous.
\end{thmint}

Our second result concerns other kinds of special conditions for Hermitian metrics. To this regard, we recall that the {\em balanced} condition in the sense of Michelsohn \cite{michelsohn} amounts to $\diff{}^*\omega=0$. The {\em pluriclosed} condition, {\itshape a.k.a.} strong K\"ahler with torsion (SKT for short), means that $\diff\diff^{\,c}\omega=0$.
Moreover, the {\em Gauduchon} condition \cite{gauduchon-excentricite} refers to $\diff{}^*\q=0$, equivalently, $\diff\diff^{\,c}\omega^{m-1}=0$.

\begin{thmint}[see Corollary \ref{cor:balanced}, Corollary \ref{cor:skt} and Proposition \ref{prop:gauduchon}]\label{thm:b}
Let $(M_{(i,n)}(\fG,\fK), \bJ, \bg)$ be one of the B\'erard-Bergery standard cohomogeneity one Hermitian manifolds.
The following three conditions are equivalent: $\bg$ is pluriclosed, $\bg$ is balanced, $\bg$ is K\"ahler.
Moreover, if $(M_{(i,n)}(\fG,\fK),\bJ)$ has singular orbits, then $\bg$ is Gauduchon if and only if it is K\"ahler.
\end{thmint}

Finally, we investigate the second-Chern-Einstein and the constant Chern-scalar curvature problems (see Section \ref{sec:Chernprob}) on this class of manifolds. Firstly, we prove a local existence and uniqueness result for second-Chern-Einstein metrics with prescribed Chern-scalar curvature (see Theorem \ref{thm:loc}), by using a method due to Malgrange \cite{malgrange} already exploited by Eschenburg and Wang \cite{eschenburg-wang} and by B\"ohm \cite{boehm98,boehm99}. Then, concerning the existence of complete solutions to the second-Chern-Einstein equations, we prove the following

\begin{thmint}[see Remark \ref{rem:homCE}, Theorem \ref{thm:IICE2}, Proposition \ref{noCE4}]\label{thm:c}
Let $(M_{(i,n)}(\fG,\fK), \bJ)$ be one of the B\'erard-Bergery standard cohomogeneity one complex manifolds.
\begin{itemize}
\item[a)]
If $(M_{(i,n)}(\fG,\fK), \bJ)$ has no singular orbits (cases $i=1$ and $i=3$), then it admits homogeneous second-Chern-Einstein metrics.
\item[b)]
If $(M_{(i,n)}(\fG,\fK), \bJ)$ has one singular orbit (case $i=2$), then it admits a complete, non-K\"ahler, second-Chern-Einstein metrics of cohomogeneity one.
\item[c)]
If $(M_{(i,n)}(\fG,\fK), \bJ)$ has two singular orbits (case $i=4$), then it does not admit any cohomogeneity one, submersion-type metrics which are second-Chern-Einstein.
\end{itemize}
\end{thmint}

Let us stress that the homogeneous second-Chern-Einstein metrics on $(M_{(1,n)}(\fG,\fK),\bJ)$ and $(M_{(3,n)}(\fG,\fK),\bJ)$, corresponding to case (a) in Theorem \ref{thm:c}, are clearly of constant Chern-scalar curvature. They include the classical example of the standard metric on the diagonal Hopf manifold, which corresponds in our notation to $M_{(3,1)}(\fG,\fK)$ with $\fG=\fSU(m)$ and $\fK=\fS(\fU(1)\times\fU(m-1))$. On the other hand, the metrics that we constructed on manifolds with singular orbits, corresponding to cases (b) and (c) in Theorem \ref{thm:c}, are {\it weakly} second-Chern-Einstein, namely, they do not have constant Chern-scalar curvature. Therefore, this brings us to investigate the constant Chern-scalar curvature problem for such manifolds. In this direction, we obtain

\begin{thmint}[see Remark \ref{rem:homCE}, Theorem \ref{thm:chcsc-M2}, Theorem \ref{thm:Chcsc-M4}]\label{thm:d}
All the B\'erard-Bergery standard cohomogeneity one complex manifolds $(M_{(i,n)}(\fG,\fK), \bJ)$ admit complete, non-K\"ahler metrics of cohomogeneity one with constant Chern-scalar curvature.
\end{thmint}

Notice that, in complex dimension $m=2$, the compact case with two singular orbits (corresponding to $i=4$) reduces to the Hirzebruch surfaces. This case has been treated by Koca and Lejmi \cite{koca-lejmi}, who proved the existence of positive constant Chern-scalar curvature of cohomogeneity one. \medskip

The paper is organized as follows. In Section \ref{prel}, we recall some basics on the Chern connection and cohomogeneity one actions. In Section \ref{sec:bb}, we recall the construction of $(M_{(i,n)}(\fG,\fK),\bJ,\bg)$ following B\'erard-Bergery, and we compute the Hermitian curvatures and torsion of such manifolds. In Section \ref{sec:metrics}, we investigate the existence of special non-K\"ahler Hermitian metrics, proving Theorem \ref{thm:a} and Theorem \ref{thm:b}.
In Section \ref{sec:chern-einstein}, we prove Theorem \ref{thm:c} and Theorem \ref{thm:d}.
Finally, in Appendix \ref{A}, we collect the detailed computations needed in Section \ref{sec:bb}, precisely to prove Propositions \ref{prop:chern-ricci} and \ref{propLee}, and in Section \ref{sec:metrics}, namely to prove Equation \eqref{ddcw}.

\bigskip

\noindent {\it Acknowledgments.} The authors are warmly grateful to Christoph B\"ohm, Simone Calamai, Liviu Ornea, Alexandra Otiman, Fabio Podest\`a, Andrea Spiro, Cristiano Spotti, Luigi Verdiani for several interesting discussions and suggestions on the topic.

\section{Preliminaries} \label{prel}
\setcounter{equation} 0

In this section, we briefly recall some facts on the Chern connection and on Hermitian manifolds with cohomogeneity one actions by holomorphic isometries.

\subsection{The Chern connection} \hfill \par

Let $(M,\bJ,\bg)$ be a connected, complete, Hermitian manifold of real dimension $\dim_{\bR}M=2m$. Let us denote by $\w \= \bg(\bJ \,\cdot\,, \cdot\,)$ its fundamental $2$-form, by $D$ its Levi-Civita connection and by $\n$ its Chern connection, which is defined by
\begin{equation} \label{defChern}
\bg(\n_AB,C) \= \bg(D_AB,C)-\tfrac12\diff\w(\bJ A,B,C) 
\end{equation}
for any $A,B,C \in \G(TM)$. Since, by the Koszul Formula,
\begin{multline} \label{KosLC}
2\bg(D_AB,C) = \eL_A(\bg(B,C))+\eL_B(\bg(A,C))-\eL_C(\bg(A,B)) \\ +\bg([A,B],C)-\bg([A,C],B)-\bg([B,C],A)
\end{multline}
and
\begin{multline} \label{dw}
\diff\w(A,B,C) = \eL_A(\bg(\bJ B,C))-\eL_B(\bg(\bJ A,C))+\eL_C(\bg(\bJ A,B)) \\ +\bg([A,B],\bJ C)+\bg([B,C],\bJ A)+\bg([C,A],\bJ B) \,\, ,
\end{multline}
it follows that
\begin{multline} \label{KosChern}
2\bg(\n_AB,C)= \eL_A(\bg(B,C)) - \eL_{\bJ A}(\bg(\bJ B,C)) +\bg([A,B],C) \\ -\bg([\bJ A,B],\bJ C) -\bg([A,C],B) +\bg([\bJ A,C],\bJ B) \,\, . 
\end{multline}
Moreover, it is well-known that the Chern connection is characterized by the following properties
\begin{equation} \label{propChern}
\n \bg = \n \bJ =0 \,\, , \qquad \bJ \t(A,B)=\t(\bJ A,B)=\t(A,\bJ B) \,\, , 
\end{equation}
where $\t(A,B) \= \n_AB -\n_BA - [A,B]$ is the {\em torsion} tensor of $\n$.

For later use, we observe the following straightforward
\begin{lemma} For any $A \in \G(TM)$ it holds that \begin{align}
\eL_A\bJ = -[\n A, \bJ] \,\, , \label{LJ1} \\
\eL_{\bJ A}\bJ = \bJ \circ \eL_A\bJ \,\, . \label{LJ2}
\end{align}
\end{lemma}
\begin{proof} If $E \in \G(\End(TM))$, then \begin{align*}
(\eL_AE)B &= \eL_A(EB)-E(\eL_AB) \\
&= \n_A(EB) -\n_{EB}A -\t(A,EB) -E(\n_AB) +E(\n_BA) +E(\t(A,B)) \\
&= -[\n A, E]B +(\n_AE)B +E(\t(A,B)) -\t(A,EB) \,\, .
\end{align*} Then, Equation \eqref{LJ1} follows by setting $E=\bJ$ and using Equation \eqref{propChern}. On the other hand, from Equation \eqref{propChern} and Equation \eqref{LJ1} we get
$$
\eL_{\bJ A}\bJ = -[\n \bJ A, \bJ] = -[\bJ\n A, \bJ] = -\bJ [\n A, \bJ] = \bJ \circ \eL_A\bJ
$$
and so the thesis follows. \end{proof}

A real vector field $A \in \G(TM)$ is {\it holomorphic} if $\eL_A\bJ=0$. Hence, we get

\begin{corollary} Let $A \in \G(TM)$. Then it holds:
\begin{itemize}
\item[$\bcdot$] $A$ is holomorphic if and only if $$\n_{\bJ B}A=\bJ\n_BA \quad \text{ for any } B \in \G(TM) \,\, ;$$
\item[$\bcdot$] $A$ is holomorphic if and only if $\bJ A$ is holomorphic.
\end{itemize}
\end{corollary}

Finally, we set $\diff^{\,c} \= \bJ^{-1} \! \circ \diff\, \circ \bJ$, so that
$$
\diff = \p+\bar{\p} \,\, , \quad \diff^{\,c} = -\sqrt{-1}(\p-\bar{\p}) \,\, , \quad \diff\diff^{\,c}=2\sqrt{-1}\p\bar{\p}
$$
and, for any smooth function $f: M \to \bR$, we denote by $\D_{\bg}^{\rm Ch}f \= \bg(\diff\diff^{\,c}\!f,\w)$ the {\it Chern-Laplacian} of $f$.

\subsection{Second-Chern-Einstein metrics} \label{sec:Chernprob} \hfill \par

Let $(M,\bJ,\bg)$ be a connected, complete, Hermitian manifold of real dimension $\dim_{\bR}M=2m$. We recall that, by the lack of symmetries of the Chern-curvature
$$
R^{\rm Ch}(\bg)(A,B) \= \n_{[A,B]} - [\n_A,\n_B] \,\, ,
$$
we have (at least) two ways to trace the Ricci tensor. We call {\it first Chern-Ricci curvature} the tensor defined by
$$
\Ric^{\rm Ch[1]}(\bg)(A,B)_x \= \sum_{e_\alpha} \bg_x\big(R^{\rm Ch}(\bg)_x(A_x,\bJ B_x)\,e_{\alpha},\,\bJ e_\alpha\big) \,\, ,
$$
where $(e_\alpha, \bJ e_\alpha)$ is a $(\bJ,\bg)$-unitary frame for the tangent space at $x$. Similarly, we call {\it second Chern-Ricci curvature} the tensor defined by
$$
\Ric^{\rm Ch[2]}(\bg)(A,B)_x \= \sum_{e_\alpha} \bg_x(R^{\rm Ch}(\bg)_x(e_\alpha,\bJ e_\alpha)\,A_x,\,\bJ B_x) \,\, .
$$
Finally, the {\it Chern-scalar curvature} is the function given by
$$
\scal^{\rm Ch}(\bg)(x) \= 2 \sum_{e_\alpha}\Ric^{\rm Ch[i]}(\bg)_x(e_\alpha,e_\alpha) \,\, , \quad i=1,2 \,\, .
$$
We remark that, according to our notation, when $\bg$ is K\"ahler it holds that $\Ric^{\rm Ch[1]}(\bg)=\Ric^{\rm Ch[2]}(\bg)=\Ric(\bg)$ and $\scal^{\rm Ch}(\bg)=\scal(\bg)$, where $\Ric(\bg)$ and $\scal(\bg)$ denote the Riemannian Ricci tensor and the Riemannian scalar curvature of $\bg$, respectively. \smallskip

This yields to the following

\begin{definition} \label{def:CE}
Let $i \in \{1,2\}$. The metric $\bg$ is said to be {\em weakly} (respectively, {\em strongly}) {\em $i^{\rm th}$-Chern-Einstein} if there exists $\l \in \cC^{\infty}(M, \bR)$ (respectively, $\l \in \bR$) such that
$$
\Ric^{{\rm Ch}[i]}(\bg)=\tfrac{\l}{2m}\bg \,\, .
$$
\end{definition}

We stress that the first-Chern-Einstein problem is basically understood for {\em compact} complex manifolds $X=(M,\bJ)$, see \cite{angella-calamai-spotti-2, tosatti}. Indeed, strongly first-Chern-Einstein metrics with non-zero Chern-scalar curvature are K\"ahler-Einstein. Moreover, by conformal methods, if a Hermitian metric $\bg$ is weakly first-Chern-Einstein with non-identically-zero Chern-scalar curvature, then it is conformal to a K\"ahler metric in the class $\pm c_1(X)$ (see \cite[Theorem A]{angella-calamai-spotti-2}). Finally, compact complex manifolds with first Bott-Chern class $c_1^{BC}(X)=0$ are the so-called {\em non-K\"ahler Calabi-Yau manifolds} \cite{tosatti} and always admit Chern-Ricci flat metrics \cite[Theorem 1.2]{tosatti}, see also \cite[Corollary 2]{tosatti-weinkove}.

On the other hand, the second-Chern-Einstein problem seems to be geometrically more appealing, see {\itshape e.g.} \cite{gauduchon-ivanov, streets-tian, podesta, angella-calamai-spotti-2}. Note that a Hermitian metric $\bg$ on $X=(M,\bJ)$ is second-Chern-Einstein according to Definition \ref{def:CE} if and only if the induced Hermitian metric $h(V,W) \= \bg(V,\ol{W})$ on the holomorphic tangent bundle $T^{1,0}X$ is Hermite-Einstein by taking trace with itself \cite{gauduchon}.

\begin{rem} \label{rmk:conf-invariant}
We observe that the second-Chern-Einstein condition is satisfied by a Hermitian metric $\bg$ if and only if it is satisfied by all the metrics in its conformal class (see \cite{gauduchon}), since for any smooth function $f: M \to \bR$ it holds that
$$
\Ric^{{\rm Ch}[2]}(e^f\bg) = \Ric^{{\rm Ch}[2]}(\bg)-(\D_{\bg}^{\rm Ch}f)\bg \,\, .
$$
We remark that this is strongly different from the Riemannian analogue, {\it i.e.} the Einstein condition. On the other hand, we stress that a K\"ahler metric is second-Chern-Einstein if and only if it is Einstein.
\end{rem}

Note that, up to our knowledge, a non-K\"ahler example of second-Chern-Einstein metric on a compact complex manifold of complex dimension $m > 2$ with negative scalar curvature is still missing.

\subsection{Cohomogeneity one group actions on Hermitian manifolds} \hfill \par

Let us consider a compact, connected real Lie group $\fG$ which acts effectively by holomorphic isometries on $(M,\bJ,\bg)$ with cohomogeneity one \cite{mostert, berard-bergery, besse, podesta-spiro}. Then, the orbit space $\W \= \fG \backslash M$ is homeomorphic to one of the following:
$$
\text{(i)}\,\, \W \simeq \bR \,\, , \qquad
\text{(ii)}\,\, \W \simeq [0,+\infty) \,\, , \qquad
\text{(iii)}\,\, \W \simeq S^1 \,\, , \qquad
\text{(iv)}\,\, \W \simeq [0,\pi] \,\, .
$$
Up to homothety, we can choose a unit speed geodesic $\g:\ol{I} \to M$ which intersects orthogonally any $\fG$-orbit \cite[Section 2.7]{berard-bergery}, where the interval $I \subset \bR$ is defined as
$$
\text{(i)}\,\, I \= \bR \,\, , \qquad \text{(ii)}\,\, I \= (0,+\infty) \,\, , \qquad \text{(iii)}\,\, I \= ({-}\pi,\pi) \,\, , \qquad \text{(iv)}\,\, I \= (0,\pi) \,\, .
$$
Then, for $r \in I$, the orbit $\eS_r \= \fG \cdot \g(r)$ is principal and can be identified with a fixed homogeneous space by means of the $1$-parameter family of $\fG$-equivariant diffeomorphisms
$$
\phi_r: \fG/\fH \to \eS_r \,\, , \quad \phi_r(a\fH) \= a \cdot \g(r) \,\, ,
$$
where $\fH \subset \fG$ is a closed subgroup. For $r \in \p I$, the following cases occur: \begin{itemize}
\item[(i)] $\p I = \emptyset$ and all the orbits are principal;
\item[(ii)] $\p I = \{0\}$, the orbit $\eS_0 \= \fG \cdot \g(0)$ is non-principal and $\fG$-equivariantly diffeomorphic to a homogeneous space $\fG/\fL$, where $\fH \subset \fL \subset \fG$ is an intermediate subgroup and $\fL / \fH$ is a sphere, see {\itshape e.g.} \cite[Section 2.12]{berard-bergery};
\item[(iii)] $\p I = \{\pm\pi\}$, the orbits $\eS_{\pm\pi} \= \fG \cdot \g(\pm\pi)$ are principal but in general $\g(-\pi) \neq \g(\pi)$, see {\itshape e.g.} \cite[Section 2.10]{berard-bergery};
\item[(iv)] $\p I = \{0,\pi\}$, the orbits $\eS_0 \= \fG \cdot \g(0)$, $\eS_{\pi} \= \fG \cdot \g(\pi)$ are non-principal and $\fG$-equivariantly diffeomorphic to two homogeneous space $\fG/\fL_{\pm}$, with $\fH \subset \fL_- \cap \fL_+$ and $\fL_{\pm} / \fH$ are spheres, see {\itshape e.g.} \cite[Section 2.13]{berard-bergery}.
\end{itemize}

The subset $M^{\rm reg} \= \bigcup_{r \in I} \eS_r$ of regular points is an open, dense submanifold of $M$ which projects onto $I$ by means of the canonical projection $M \to \W$ and the restricted Riemannian metric splits as
$$
\bg|_{M^{\rm reg}}=\diff r^2 + \bg|_{\eS_r} \,\, ,
$$
where $r$ is the coordinate on $I$. We fix an $\Ad(\fG)$-invariant inner product $Q$ on the Lie algebra $\gg \= \Lie(\fG)$ and a $Q$-orthogonal decomposition $\gg=\gh+\gm$, with $\gh \= \Lie(\fH)$. We consider identified $\gm \simeq T_{e\fH}\fG/\fH$ by means of the evaluation map $X \mapsto \frac{d}{ds}\exp(sX)\fH \big|_{s=0}$, where $\exp: \gg \to \fG$ denotes the Lie exponential map of $\fG$. We also identify any $\fG$-invariant tensor field on $\fG/\fH$ with the corresponding $\Ad(\fH)$-invariant tensor on $\gm$ in the usual way. \smallskip

Firstly, we define the $1$-parameter family $(g_r) \subset \Sym^2(\gm^*)^{\Ad(\fH)}$ by
$$
g_r \= (\phi_r)^*(\bg|_{\eS_r}) \,\, .
$$
We also set
$$
T_r \= \frac{(\diff \phi_r)_{e\fH}^{-1}\big(\bJ_{\g(r)}\dot{\g}(r)\big)}{\big|(\diff \phi_r)_{e\fH}^{-1}\big(\bJ_{\g(r)}\dot{\g}(r)\big)\big|_Q} \in \gm \,\, , \quad \theta_r \= Q(T_r,\,\cdot\,)|_{\gm} \in \gm^* \,\, .
$$
Then, for any $r \in I$, the complex structure $\bJ$ induces a linear complex structure $J_r$ on $\gp_r \= \ker(\theta_r) \subset \gm$ by setting
$$
J_r : \gp_r \to \gp_r \,\, , \quad J_r \= (\diff \phi_r)_{e\fH}^{-1} \circ \bJ_{\g(r)} \circ (\diff \phi_r)_{e\fH}|_{\gp_r} \,\, .
$$
The integrability of $\bJ$ implies that, for any $X,Y \in \gp_r$, 
\begin{equation} \begin{gathered}
\!\, [J_rX,Y]_{\gm}+[X,J_rY]_{\gm} \in \gp_r \,\, , \\ J_r\big([J_rX,Y]_{\gm}+[X,J_rY]_{\gm}\big)=[J_rX,J_rY]_{\gm}-[X,Y]_{\gm} \,\, .
\end{gathered} \end{equation}
Moreover, since the subgroup $\fH \subset \fG$ leaves any point of the geodesic $\g$ fixed and $\fG$ acts by holomorphic isometries, it follows that
\begin{equation}
[\gh,T_r]=0 \quad \text{ for any } r \in I \label{hT}
\end{equation}
and so $(\theta_r,J_r)$ is a $1$-parameter family of $\fG$-invariant CR structures on $\fG/\fH$, see {\itshape e.g.} \cite{alekseevsky-spiro}. 

We recall now the following definition \cite{podesta-spiro}: $\fG/\fH$ is said to be {\it ordinary} if $\fG$ is semisimple, the normalizer $\fK\=\fN_{\fG}(\fH^{\zero})$ of the connected component $\fH^{\zero}$ of $\fH$ is the centralizer of a torus and $\dim\fK=1+\dim\fH$. This implies that: \begin{itemize}
\item[$\bcdot$] $T \equiv T_r$ and $\gp \equiv \gp_r$ do not depend on $r \in I$;
\item[$\bcdot$] the Lie algebra $\gk \= \Lie(\fK)$ splits $Q$-orthogonally as $\gk = \gh \oplus \bR T$.
\end{itemize}
Moreover, the complex structure $\bJ$ is said to be {\it projectable} if each $J_r$ is $\Ad(\fK)$-invariant. In this case, $(J_r)_{r \in I}$ is mapped onto a $1$-parameter family of $\fG$-invariant complex structures on the flag manifold $\fG / \fK$. Since the set of invariant complex structures on a flag manifold is discrete \cite{nishiyama, alekseevsky-perelomov}, it follows that $J \equiv J_r$ is constant.

\begin{definition}[\cite{podesta-spiro, alekseevsky-zuddas-1}]\label{def:standard}
The Hermitian manifold $(M,\bJ,\bg)$ acted by $\fG$ is said to be {\it standard} if the following conditions are satisfied: \begin{itemize}
\item[$\bcdot$] the principal orbits are ordinary and the complex structure $\bJ$ is projectable;
\item[$\bcdot$] the non-principal orbits, if they exist, are flag manifolds with the induced complex structure.
\end{itemize} \end{definition}

In Section \ref{sec:bb}, a distinguished kind of cohomogeneity one standard Hermitian manifolds will be investigated.

\section{B\'erard-Bergery standard cohomogeneity one Hermitian manifolds}\label{sec:bb}
\setcounter{equation} 0

In this section, we consider a special class of standard cohomogeneity one Hermitian manifolds, following the construction of B\'erard-Bergery \cite{berard-bergery}. Then, we compute the Chern connection and the Chern-Ricci tensors of such manifolds.

\subsection{Chern connection of standard cohomogeneity one Hermitian manifolds}\label{sec:stcohom1} \hfill \par

Let $(M,\bJ,\bg)$ be a standard cohomogeneity one Hermitian manifolds acted effectively by holomorphic isometries by a compact, connected real Lie group $\fG$.

Hereafter, we adopt the same notation introduced in Section \ref{prel}. Notice that the complement $\gm$ of $\gh$ in $\gg$ admits a $Q$-orthogonal, $\Ad(\fH)$-invariant decomposition $\gm= \ga + \gp$, where $\ga \= \bR T$ is the trivial submodule. Since, by hypothesis, $\gp$ does not contain any trivial $\Ad(\fH)$-submodule, the metrics $g_r$ induced by $\bg$ on $\gm$ split uniquely as $$g_r = F(r)^2 Q|_{\ga {\otimes} \ga} + g_r|_{\gp {\otimes} \gp}$$ by means of the Schur Lemma, where $F: I \to \bR$ is a smooth, positive function, possibly satisfying some appropriate boundary condition.

From now on, given $V \in \gg$, we denote by $V^* \in \G(TM)$ the fundamental vector field on $M$ associated to $V$, that is $V^*_p \= \frac{d}{ds}\exp(sV)\cdot p \big|_{s=0}$, and we also set
\begin{equation} \label{defN}
N\=F\tfrac{\p}{\p r} \,\, .
\end{equation}
Notice that, by construction,
\begin{equation} \label{JNT}
(\bJ N)_{\g(r)} = T^*_{\g(r)}
\end{equation}
and, for any $V,W \in \gg$, it holds
\begin{equation} \label{brack*}
[V^*,W^*]=-[V,W]^* \,\, , \quad [N,V^*]=0 \,\, .
\end{equation}
We observe also that, for any $X \in \gp$ and for any $Y_1,Y_2 \in \gm$, we have
\begin{equation} \label{gJ*}
\bg(Y_1^*,Y_2^*)_{\g(r)} = g_r(Y_1,Y_2) \,\, , \quad (\bJ X^*)_{\g(r)} = (JX)^*_{\g(r)} \,\, .
\end{equation}
By hypothesis, the fundamental vector fields $V^*$ are holomorphic, {\itshape i.e.}
\begin{equation} \label{V*J}
[V^*,\bJ A] = \bJ [V^*,A] \quad \text{ for any $A \in \G(TM)$ } \,\, .
\end{equation}

By using Equations \eqref{KosLC} and \eqref{dw}, one can directly obtain the explicit formulas for the Levi-Civita connection and the $3$-form $\diff\w$. Here and in the following statements, we consider the context and assume the notation as above.

\begin{prop} \label{propLC}
Let $X, Y, Z \in \gp$. Then, the non-vanishing components of the Levi-Civita connection are given by
\begin{align*}
2\bg\big(D_{Y^*}X^*,Z^*\big)_{\g(r)} &= g_r([X,Y],Z)+g_r([Z,X],Y)+g_r([Z,Y],X) \,\,,\\
2\bg\big(D_{Y^*}X^*,T^*\big)_{\g(r)} &= g_r([X,Y],T)+g_r([T,X],Y)+g_r([T,Y],X) \,\,,\\
2\bg\big(D_{Y^*}X^*,N\big)_{\g(r)} &= -F(r)\tfrac{\p}{\p r}(g_r(X,Y)) \,\,,\\
2\bg\big(D_{T^*}X^*,Z^*\big)_{\g(r)} &= -g_r([X,Z],T)-g_r([T,X],Z)-g_r([T,Z],X) \,\,,\\
2\bg\big(D_NX^*,Z^*\big)_{\g(r)} &= F(r)\tfrac{\p}{\p r}(g_r(X,Z)) \,\,,\\
2\bg\big(D_{Y^*}T^*,Z^*\big)_{\g(r)} &= -g_r([Y,Z],T) +g_r([T,Y],Z)-g_r([T,Z],Y) \,\,,\\
2\bg\big(D_{T^*}T^*,N\big)_{\g(r)} &= -2F'(r)F(r)^2 \,\,,\\
2\bg\big(D_NT^*,T^*\big)_{\g(r)} &= 2F'(r)F(r)^2 \,\,,\\
2\bg\big(D_{Y^*}N,Z^*\big)_{\g(r)} &= F(r)\tfrac{\p}{\p r}(g_r(Y,Z)) \,\,,\\
2\bg\big(D_{T^*}N,T^*\big)_{\g(r)} &= 2F'(r)F(r)^2 \,\,,\\
2\bg\big(D_NN,N\big)_{\g(r)} &= 2F'(r)F(r)^2 \,\,.
\end{align*}
\end{prop}

\begin{prop} \label{propdw}
Let $X,Y,Z \in \gp$. Then, the $3$-form $\diff\w$ is given by \begin{align*}
\diff\w(X^*,Y^*,Z^*)_{\g(r)} &= g_r([X,Y],JZ) +g_r([Y,Z],JX) +g_r([Z,X],JY) \,\, , \\
\diff\w(X^*,Y^*,T^*)_{\g(r)} &= g_r([T,X],JY) +g_r([T,JY],X) \,\, ,\\
\diff\w(X^*,Y^*,N)_{\g(r)} &= F(r)\tfrac{\p}{\p r}\big(g_r(JX,Y)\big) +g_r([X,Y],T) \,\, ,\\
\diff\w(X^*,T^*,N)_{\g(r)} &= 0 \,\, .
\end{align*}
\end{prop}

By combining Proposition \ref{propLC} and Proposition \ref{propdw}, and by Equation \eqref{defChern}, we obtain explicit formulas for the Chern connection along the geodesic $\g(r)$. More precisely

\begin{prop} \label{propCh}
Let $X,Y,Z \in \gp$. Then the non-vanishing components of the Chern connection are
\begin{align*}
2\bg\big(\n_{Y^*}X^*,Z^*\big)_{\g(r)} &= g_r([X,Y],Z) +g_r([X,JY],JZ) +g_r([Z,Y],X) -g_r([Z,JY],JX) \,\,, \\
2\bg\big(\n_{Y^*}X^*,T^*\big)_{\g(r)} &= g_r([X,Y],T) \,\,, \\
2\bg\big(\n_{Y^*}X^*,N\big)_{\g(r)} &= g_r([X,JY],T) \,\,, \\
2\bg\big(\n_{T^*}X^*,Z^*\big)_{\g(r)} &= F(r)\tfrac{\p}{\p r}\big(g_r(JX,Z)\big) -g_r([T,X],Z) -g_r([T,Z],X) \,\,, \\
2\bg\big(\n_NX^*,Z^*\big)_{\g(r)} &= F(r)\tfrac{\p}{\p r}(g_r(X,Z)) +g_r([T,JX],Z) +g_r([T,Z],JX) \,\,,\\
2\bg\big(\n_{Y^*}T^*,Z^*\big)_{\g(r)} &= 2g_r([T,Y],Z) -g_r([Y,Z],T) \,\,, \\
2\bg\big(\n_{T^*}T^*,N\big)_{\g(r)} &= -2F'(r)F(r)^2 \,\,, \\
2\bg\big(\n_NT^*,T^*\big)_{\g(r)} &= +2F'(r)F(r)^2 \,\,, \\
2\bg\big(\n_{Y^*}N,Z^*\big)_{\g(r)} &= g_r([JY,Z],T) \,\,,\\
2\bg\big(\n_{T^*}N,T^*\big)_{\g(r)} &= 2F'(r)F(r)^2 \,\,, \\
2\bg\big(\n_NN,N\big)_{\g(r)} &= 2F'(r)F(r)^2 \,\,.
\end{align*}
\end{prop}

As a direct consequence of Proposition \ref{propCh} and Equation \eqref{LJ1}, we get 

\begin{cor} \label{corLNJ}
It holds that $(\eL_{N}\bJ)_{\g(r)} = 0$, {\itshape i.e.} $[N,\bJ A]_{\g(r)} = \bJ[N, A]_{\g(r)}$ for any $A \in \G(TM)$.
\end{cor}

We can also characterize the K\"ahler metrics as follows. By Proposition \ref{propdw} it follows that
$$
\diff\w(X^*,Y^*,T^*)_{\g(r)}=0 \quad \text{ for any } X,Y \in \gp
$$
if and only if the restriction $g_r|_{\gp{\otimes}\gp}$ is $\Ad(\fK)$-invariant for any $r \in I$. This is equivalent to say that the metrics $(g_r)$ on $\fG/\fH$ are {\em of submersion-type} with respect to the homogeneous circle bundle $\fK/\fH \to \fG/\fH \to \fG/\fK$, namely, they induce metrics on the base such that the projection is a Riemannian submersion.
Moreover, if $g_r|_{\gp{\otimes}\gp}$ is $\Ad(\fK)$-invariant, then the condition
$$
\diff\w(X^*,Y^*,Z^*)_{\g(r)}=0 \quad \text{ for any } X,Y,Z \in \gp
$$
holds true if and only if the induced metrics on the flag manifold $\fG/\fK$ are K\"ahler. Let us fix now a $Q$-orthogonal, $\Ad(\fK)$-invariant, irreducible decomposition
\begin{equation} \label{decp}
\gp= \gp_1+{\dots}+\gp_{\ell} \,\, .
\end{equation}
Since flag manifolds are {\em equal rank homogeneous spaces}, namely ${\rm rank}(\fK)={\rm rank}(\fG)$, it follows that their isotropy representations are always {\em monotypic}, namely $\gp_i \not\simeq \gp_j$ for any $1 \leq i < j \leq \ell$. Hence, the decomposition \eqref{decp} is unique (up to order) and, by the Schur Lemma, the metrics $g_r$ splits uniquely as
$$
g_r = F(r)^2 Q|_{\ga {\otimes} \ga} +h_1(r)^2 Q|_{\gp_1 {\otimes} \gp_1} + {\dots} + h_{\ell}(r)^2 Q|_{\gp_{\ell} {\otimes} \gp_{\ell}} \,\, .
$$
Then, the condition
$$
\diff\w(X^*,Y^*,N)_{\g(r)}=0 \quad \text{ for any } X, Y \in \gp
$$
holds true if and only if
\begin{equation} \label{[T,X]}
\ad(T)|_{\gp_i} = -2\tfrac{h_i(r)h'_i(r)}{F(r)} J \quad \text{ for any $1 \leq i \leq \ell$ } \,\, .
\end{equation}

\subsection{Construction of B\'erard-Bergery manifolds} \label{sec:construction} \hfill \par

Let us assume that $P = P(\fG,\fK)\= \fG/\fK$ is a simply-connected, irreducible compact Hermitian symmetric space, {\itshape i.e.} $\fG$ is a connected, compact, simple Lie group and $\fK \subset \fG$ is a maximal, connected, compact subgroup with center isomorphic to the circle group \cite[Theorem 6.1]{helgason}. Set $\dim_{\bR}P=2(m{-}1)$, with $m\geq2$, and let $p \in \bN$ be the unique positive integer such that $p^{-1}c_1(P)$ is an indivisible (positive) class in the cohomology group $H^2(P;\bZ)$, see \cite[Chapter 5, Section 16]{borel-hirzebruch-1}. In the following, we list the possibilities for $P$, following \cite[Section 7.102 and Section 9.124]{besse}.
$$\begin{array}{c|c|c|c}
P & m & p & \text{conditions} \\
\hline
\fSU(k_1+k_2)/\fS(\fU(k_1){\times}\fU(k_2)) & k_1k_2+1& k_1+k_2 & k_2\geq k_1 \geq 1 \\
\fSO(2k)/\fU(k) & \tfrac{k(k-1)}{2}+1 & 2(k-1) & k\geq5 \\
	\fSp(k)/\fU(k) & \tfrac{k(k+1)}{2}+1 & k+1 & k\geq2 \\
\fSO(k+2)/(\fSO(2){\times}\fSO(k)) & k+1 & k & k\geq5 \\
\fE_6/\fSO(2)\fSpin(10) & 17 & 12 & - \\
\fE_7/\fSO(2)\fE_6 & 28 & 18 & - \\
\end{array}$$

Denote by $\fH \= [\fK,\fK]$ the commutator of $\fK$. By hypothesis, there exists an integer $s\geq 1$ such that $\fH \cap \fZ(\fK) \simeq \bZ_s$. Fix an isomorphism $\imath: \fU(1) \to \fZ(\fK)$ and consider the generator $\a \= \imath(e^{\sqrt{-1}\frac{2\pi}{s}})$ of $\fH \cap \fZ(\fK)$. Then, consider the right action of $\bZ_s$ on $\fH \times \fU(1)$ given by $(h,z)\cdot j \= \big(h\a^j,e^{-\sqrt{-1}\frac{2j\pi}{s}}z\big)$ and take the quotient $\fH{\cdot}\fU(1) \= \fH \times_{\bZ_s}\!\! \fU(1)$. We can consider identified $\fK = \fH{\cdot}\fU(1)$ by means of a Lie group isomorphism (see {\itshape e.g.} \cite[Chapter 0, Theorem 6.9]{bredon}) and, for any positive integer $n$, we take the representation
$$
\rho_n : \fK \to \fU(1) \,\, , \quad \rho_n([h,z]) \= z^{-sn} \,\, .
$$
Then, consider the associated bundle
$$
\S_n = \S_n(\fG,\fK) \= \fG \times_{\rho_n} \fU(1)
$$
with the projection $\pi_n: \S_n \to P$ given by $\pi_n([a,z])\=a\fK$. Notice that $\fG$ acts transitively on the left on $\S_n$ by $\tilde{a} \cdot [a,z] \= [\tilde{a}a,z]$ and the stabilizer of $[e,1]$ is the subgroup $\ker(\rho_n) = \fH {\times} \bZ_n \subset \fK$, where
$$
\fH {\times} \bZ_n \simeq \fH \times_{\bZ_s}\!\! \bZ_{sn} \subset \fH{\cdot}\fU(1) = \fK \quad \text{ by } \quad (h,j) \mapsto \big[h,e^{\sqrt{-1}\frac{2j\pi}n}\big] \,\, .
$$
We are ready to construct the family $M_{(i,n)}(\fG,\fK)$, with $1 \leq i \leq 4$ and $n \in \bN$, of standard cohomogeneity one Hermitian manifolds in the following way.
\begin{itemize}
\item[i)] We set $M_{(1,n)}(\fG,\fK) \= \S_n \times \bR$ and we let $\fG$ act on $M_{(1,n)}$ via $a \cdot (x, r) \= (a\cdot x,r)$. The orbit space is $\W=\bR$.
\item[ii)] We let $M_{(2,n)}(\fG,\fK) \= \S_n \times_{\fU(1)} \bC$ be the homogeneous complex line bundle over $P$ associated to $\pi_n: \S_n \to P$ by means of the standard action of $\fU(1)$ on $\bC$. Then, $M_{(2,n)}(\fG,\fK)$ is equivariantly diffeomorphic to the quotient of $\S_n \times [0,+\infty)$ by the fibration $\S_n \times \{0\} \to P$, on which $\fG$ acts by left multiplication on the first factor.
\item[iii)] We set $M_{(3,n)}(\fG,\fK) \= \S_n \times S^1$ and we let $\fG$ act on $M_{(3,n)}(\fG,\fK)$ via $a \cdot (x, z) \= (a\cdot x,z)$. The orbit space is $\W=S^1$.
\item[iv)] We let $M_{(4,n)}(\fG,\fK) \= \S_n \times_{\fU(1)} \bC \bP^1$ be the homogeneous $\bC \bP^1$-bundle over $P$ associated to $\pi_n: \S_n \to P$ by means of the standard action of $\fU(1)$ on $\bC \bP^1$. Then, $M_{(4,n)}(\fG,\fK)$ is equivariantly diffeomorphic to the quotient of $\S_n \times [0,\pi]$ under the identification of the two boundaries by means of the fibrations $\S_n \times \{0\} \to P$ and $\S_n \times \{\pi\} \to P$, on which $\fG$ acts by left multiplication on the first factor.
\end{itemize}

We observe that the manifolds in families $M_{(3,n)}(\fG,\fK)$ and $M_{(4,n)}(\fG,\fK)$ are compact, and the manifolds in families $M_{(2,n)}(\fG,\fK)$ and $M_{(4,n)}(\fG,\fK)$ are simply connected.
Moreover, manifolds in family $M_{(4,n)}(\fG,\fK)$ are {\em almost-homogeneous} in the sense of \cite{huckleberry-snow}, {\itshape i.e.} the complexified Lie group $\fG^{\bC}$ acts on them by biholomorphisms with one open orbit, see \cite{podesta-spiro-2}. \smallskip

Let $M_{(i,n)}(\fG,\fK)$ be as above. We denote by $B$ be the Cartan-Killing form of $\gg \= \Lie(\fG)$ and we set $\gk \= \Lie(\fK)$, $\gh \= \Lie(\fH)$, $\ga \= \gz(\gk) = \Lie(\fZ(\fK))$. Then, the positive definite $\Ad(\fG)$-invariant scalar product $Q \= \tfrac1{4m}(-B)$ on $\gg$ determines a $\Ad(\fH)$-invariant, $Q$-orthogonal decomposition
$$
\gg=\underbrace{\gh + \ga}_{\gk} + \!\!\!\!\!\!\!\!\!\!\overbrace{\phantom{kiia} \gp}^{\gm} \,\, , \quad \text{ with } \,\,\, [\gk,\gk]=\gh \,\, , \,\,\, [\gh,\ga] = \{0\} \,\, , \,\,\, [\gk,\gp]\subset \gp \,\, , \quad [\gp,\gp] \subset \gk \,\, .
$$
We fix a vector $T \in \ga$ with $Q(T,T)=1$ and we pick the only $\l>0$ such that
\begin{equation} \label{Jad(T)}
J \= \l^{-1}\ad(T)|_{\gp}
\end{equation}
is a linear complex structure on $\gp$ \cite[Chapter XI, Theorem 9.6]{kobayashi-nomizu-2}. Notice that a direct computation implies
$$
-4m = B(T,T) = \Tr(\ad(T) \circ \ad(T)) = -\l^2 \Tr(\Id_{2(m{-}1)})
$$
and so
\begin{equation} \label{lambda}
\l^2=\tfrac{2m}{m-1} \,\, .
\end{equation}

\begin{remark}[{\cite[page 814]{podesta-spiro}}] \label{rem:J}
The linear complex structure $J$ on $\gp$ determines uniquely a $\fG$-invariant, projectable, complex structure $\bJ$ on $M_{(i,n)}(\fG,\fK)$.
\end{remark}

\noindent Moreover, we stress that: \begin{itemize}
\item[$\bcdot$] the restriction $Q_{\gp} \= Q|_{\gp{\otimes}\gp}$ induces a $\fG$-invariant K\"ahler-Einstein metric on the base space $P$ satisfying the equation $\Ric(Q_{\gp}) =2m Q_{\gp}$
\item[$\bcdot$] the scalar product $\tfrac{2m(m-1)n^2}{p^2}Q_{\ga}$, with $Q_{\ga}\=Q|_{\ga{\otimes}\ga}$, corresponds to the standard metric of radius $1$ on the fibres of $\pi_n: \S_n \to P$
\end{itemize}
Then, being $P$ irreducible, any $\fG$-invariant Hermitian metric $\bg$ on $(M_{(i,n)}(\fG,\fK),\bJ)$ which is of submersion-type with respect to $M_{(i,n)}(\fG,\fK) \to P$ is completely determined by two positive, smooth functions $f,h: I \to \bR$, satisfying some appropriate {\em smoothness conditions}, by means of the splitting
\begin{equation} \label{eq:metric}
g_r = \tfrac{2m(m-1)n^2}{p^2}f(r)^2Q_{\ga} +h(r)^2Q_{\gp} \,\, ,
\end{equation}
where $(g_r)$ is again the 1-parameter family of $\fG$-invariant metrics induced by $\bg$ on the principal orbits. Case by case, the smoothness conditions are the following (see {\it e.g.} \cite[page 7]{verdiani-ziller} and \cite[page 39]{berard-bergery}).
\begin{itemize}
\item[i)] For $i=1$, $I=\bR$ and there are no boundary conditions.
\item[ii)] For $i=2$, $I=(0,+\infty)$ and the conditions are: $f$ is the restriction of a smooth odd function on $\bR$ with $f'(0) = 1$ and $h$ is the restriction of a smooth even function on $\bR$.
\item[iii)] For $i=3$, $f,h$ need to be $S^1$-periodic.
\item[iv)] For $i=4$, $I=(0,\pi)$ and the conditions are: $f$ is the restriction of a smooth odd function on $\bR$ satisfying $f(\pi{+}r)=-f(\pi{-}r)$ with $f'(0) = 1 = -f'(\pi)$ and $h$ is the restriction of a smooth even function on $\bR$ satisfying $h(\pi{+}r)=h(\pi{-}r)$.
\end{itemize}
Conversely, any pair of smooth functions $(f,h)$ satisfying the appropriate smoothness condition uniquely defines a smooth $\fG$-invariant Hermitian metric $\bg=\bg(f,h)$, which is of submersion-type.

By Equations \eqref{[T,X]}, \eqref{lambda} and \eqref{eq:metric}, it follows that the metric $\bg(f,h)$ is K\"ahler if and only if the functions $f,h$ verify
\begin{equation} \label{eq:kahler}
h(r)h'(r)+\tfrac{mn}{p}f(r) = 0 \quad \text{ for any $r \in I$ } \,\, .
\end{equation}

\begin{rem} \label{rem:noK}
We notice that the smoothness condition in (iii) and Equation \eqref{eq:kahler} imply immediately that the complex manifolds $(M_{(3,n)}(\fG,\fK),\bJ)$ cannot admit cohomogeneity one, submersion-type K\"ahler metrics (see also \cite[Corollary 20]{alekseevsky-zuddas-1}). Actually, it holds more: it can be easily check that $\pi_1(\S_n)$ is finite, and so its first Betti number is $b_1(\S_n)=0$. In particular, this implies that $(M_{(3,n)}(\fG,\fK),\bJ)$ do not admit K\"ahler metrics at all.
\end{rem}

From now on, we will adopt the following

\begin{definition} \label{def:bb}
A {\em B\'erard-Bergery standard cohomogeneity one Hermitian manifold} is a triple $(M_{(i,n)}(\fG,\fK),\bJ,\bg)$, where $M_{(i,n)}(\fG,\fK)$ is the bundle over $P=\fG/\fK$ constructed as above, $\bJ$ is the unique $\fG$-invariant projectable complex structure on $M_{(i,n)}(\fG,\fK)$ as in Remark \ref{rem:J} and $\bg=\bg(f,h)$ is the Riemannian metric described in Equation \eqref{eq:metric}. Accordingly, any pair $(M_{(i,n)}(\fG,\fK), \bJ)$ will be called {\em B\'erard-Bergery standard cohomogeneity one complex manifold}.
\end{definition}

Let us point out that the above construction can be performed in a more general setting, {\itshape i.e.} by requiring that the base space $P=\fG/\fK$ is a K\"ahler C-space, see \cite{berard-bergery}. However, in this work, we will just focus in the case of $P$ being symmetric and irreducible.

\begin{example} \label{ex:hopf}
Consider the Hermitian symmetric space $P=\bC \bP^{m{-}1}$, corresponding to $\fG=\fSU(m)$ and $\fK=\fS(\fU(1){\times}\fU(m{-}1))$. Here, the $\Ad(\fG)$-invariant scalar product $Q(A_1,A_2)\=-\frac{1}{2}\Tr(A_1A_2)$ on the Lie algebra $\gg=\su(m)$, defined following the above normalization, induces on $P$ the Fubini-Study metric with sectional curvature satisfying $1\leq\sec\leq4$. In this case, the principal orbits are equivariantly diffeomorphic to the lens space $\S_n = \bZ_n \backslash S^{2m-1}$, where $\bZ_n$ acts on $S^{2m-1} \subset \bC^m$ via $k \cdot z \= e^{-i\frac{2k\pi}n}z$, and
$$\begin{aligned}
&M_{(1,n)}(\fG,\fK) = \bZ_n \backslash S^{2m-1} \times \bR \,\, , &\quad& M_{(2,n)}(\fG,\fK) = \cO_{\bC \bP^{m-1}}(-n) \,\, , \\
&M_{(3,n)}(\fG,\fK) = \bZ_n \backslash S^{2m-1} \times S^1 \,\, , &\quad& M_{(4,n)}(\fG,\fK) = \bP\big(\cO_{\bC \bP^{m-1}} \oplus \cO_{\bC \bP^{m-1}}(-n)\big) \,\, .
\end{aligned}$$
Here, we denoted by $\cO_{\bC \bP^{m-1}}$ the trivial line bundle over $\bC \bP^{m-1}$, by $\cO_{\bC \bP^{m-1}}(-1)$ the tautological line bundle and by $\cO_{\bC \bP^{m-1}}(-n) \= \cO_{\bC \bP^{m-1}}(-1)^{\otimes n}$ . In particular: \begin{itemize}
\item[$\bcdot$] if $(i,n)=(3,1)$, then we get a diagonal Hopf manifold;
\item[$\bcdot$] if $(i,n)=(4,1)$, then we get the connected sum $\bC\bP^m \# \ol{\bC\bP}{}^m$;
\item[$\bcdot$] if $m=2$ and $i=4$, then we get all the Hirzebruch surfaces.
\end{itemize}
\end{example}

\begin{remark}
Let us consider $P=\bC \bP^{m{-}1}$ and assume that $m\geq 3$. Since the isotropy representation of the odd sphere $\S_1 = S^{2m-1}=\fSU(m)/\fSU(m-1)$ is monotypic, all the cohomogeneity one Hermitian metrics on $(M_{(i,n)}(\fG,\fK),\bJ)$ are of submersion type with respect to the fibration $M_{(i,n)}(\fG,\fK) \to \bC\bP^{m{-}1}$. This fact does not hold true in general. For example, starting from the Grassmannian $P=\widetilde{\rm Gr}(2,\bR^{m+1})= \fSO(m+1)/(\fSO(2){\times}\fSO(m-1))$ of the oriented $2$-planes in $\bR^{m+1}$, we get that $\S_1 = V(2,\bR^{m+1}) = \fSO(m+1)/\fSO(m-1)$ is the Stiefel manifold of the orthonormal $2$-frames in $\bR^{m+1}$, whose isotropy representation contains two equivalent irreducible summands.
\end{remark}

\begin{remark} \label{admissible}
Notice that, by means of the action of $\fG$, there is a bijective correspondence between the set of smooth functions $\tilde{\f}: M_{(i,n)}(\fG,\fK) \to \bR$ and the set of functions $\f: I \to \bR$ satisfying the appropriate smoothness condition:
\begin{itemize}
\item[i)] if $i=1$, $I=\bR$ and $\f$ is smooth;
\item[ii)] if $i=2$, $I=(0,+\infty)$ and $\f$ is the restriction of a smooth even function on $\bR$;
\item[iii)] if $i=3$, $\f$ is smooth and $S^1$-periodic;
\item[iv)] if $i=4$, $I=(0,\pi)$ and $\f$ is the restriction of a smooth even function on $\bR$ satisfying $\f(\pi{+}r)=\f(\pi{-}r)$.
\end{itemize}
From now on, any function $\f: I \to \bR$ satisfying the appropriate smoothness condition will be called {\it admissible}.
\end{remark}

\begin{remark} \label{confchange}
Let $\bg=\bg(f,h)$ a cohomogeneity one, submersion-type metric on a manifold $(M_{(i,n)}(\fG,\fK),\bJ)$ and $\f: I \to \bR$ a positive, admissible function. Then, the metric $\f^2\bg$ is still a cohomogeneity one, submersion-type metric on $(M_{(i,n)}(\fG,\fK),\bJ)$ and $\f^2\bg = \bg(\hat{f}_{\f},\hat{h}_{\f})$, with
$$
\hat{f}_{\f}(r) \= \xi'_{\f}(\xi^{-1}_{\f}(r))f(\xi^{-1}_{\f}(r)) \,\, , \quad \hat{h}_{\f}(r) \= \xi'_{\f}(\xi^{-1}_{\f}(r))h(\xi^{-1}_{\f}(r)) \,\, ,
$$
where $\xi_{\f}(r) \= \int_0^r\f(t)\diff t $. However we stress that, even if $\bg$ is complete, in general the conformal metric $\f^2\bg$ is not.
\end{remark}

\subsection{Curvature and torsion computations for B\'erard-Bergery manifolds} \label{sec:curvature-bb}\hfill \par

We begin this section by listing the Levi-Civita connection and the Riemannian Ricci tensor of the manifolds $(M_{(i,n)}(\fG,\fK),\bJ,\bg)$. By straightforward computations, from Proposition \ref{propLC} and Equations \eqref{eq:metric}, \eqref{Jad(T)} we get

\begin{prop} \label{prop:lc-connetion}
Let $X,Y,Z \in \gp$. Then
\begin{align*}
D_{Y^*}X^*\big|_{\g(r)} &= \tfrac{\l}2\big(Q(JX,Y)T^*_{\g(r)} -\tfrac{p}{mn}\tfrac{h(r)h'(r)}{f(r)}Q(X,Y)N_{\g(r)}\big) \,\, , \\
D_{T^*}X^*\big|_{\g(r)} &=-\tfrac2{\l} \big(\tfrac{mn}p\big)^2 \tfrac{f(r)^2}{h(r)^2}(JX)^*_{\g(r)} \,\, , \\
D_{N}X^*\big|_{\g(r)} &= \tfrac{2mn}{\l p} f(r)\tfrac{h'(r)}{h(r)}X^*_{\g(r)} \,\,, \\
D_{Y^*}T^*\big|_{\g(r)} &= \big(\l-\tfrac2{\l} \big(\tfrac{mn}p\big)^2\tfrac{f(r)^2}{h(r)^2}\big)(JY)^*_{\g(r)} \,\, , \\
D_{T^*}T^*\big|_{\g(r)} &= -\tfrac{2mn}{\l p} f'(r)N_{\g(r)} \,\, , \\
D_{N}T^*\big|_{\g(r)} &= \tfrac{2mn}{\l p} f'(r)T^*_{\g(r)} \,\, , \\
D_{Y^*}N\big|_{\g(r)} &= \tfrac{2mn}{\l p} f(r)\tfrac{h'(r)}{h(r)}Y^*_{\g(r)} \,\, , \\
D_{T^*}N\big|_{\g(r)} &= \tfrac{2mn}{\l p} f'(r)T^*_{\g(r)} \,\, , \\
D_{N}N\big|_{\g(r)} &= \tfrac{2mn}{\l p} f'(r)N_{\g(r)} \,\, .
\end{align*}
\end{prop}

Moreover, from \cite[Proposition 1.14]{grove-ziller}, we directly obtain

\begin{prop}\label{prop:riem-ricci} Let $X \in \gp$ with $Q(X,X)=1$. Then the Riemannian Ricci tensor is given by
\begin{align*}
\Ric(\bg)(N,N)_{\g(r)} &= \tfrac{2m(m-1)n^2}{p^2}f(r)^2\Big({-}\tfrac{f''(r)}{f(r)} {-}2(m{-}1)\tfrac{h''(r)}{h(r)}\Big) \,\, , \\
\Ric(\bg)(T^*,T^*)_{\g(r)} &= \tfrac{2m(m-1)n^2}{p^2}f(r)^2\Big({-}\tfrac{f''(r)}{f(r)} {-}2(m{-}1)\tfrac{f'(r)}{f(r)}\tfrac{h'(r)}{h(r)}{+}2(m-1)\big(\tfrac{mn}p\big)^2\tfrac{f(r)^2}{h(r)^4}\Big) \,\, , \\
\Ric(\bg)(X^*,X^*)_{\g(r)} &= h(r)^2\Big({-}\tfrac{h''(r)}{h(r)} {-}\tfrac{f'(r)}{f(r)}\tfrac{h'(r)}{h(r)}{-}(2m{-}3)\tfrac{h'(r)^2}{h(r)^2} {-}2\big(\tfrac{mn}p\big)^2\tfrac{f(r)^2}{h(r)^4} {+}\tfrac{2m}{h(r)^2}\Big)
\end{align*}
\noindent and
$$
\Ric(\bg)(N,T^*)_{\g(r)} = \Ric(\bg)(N,X^*)_{\g(r)} = \Ric(\bg)(T^*,X^*)_{\g(r)} = 0 \,\, .
$$
Furthermore, the Riemannian scalar curvature is
\begin{align*}
\scal(\bg)(r) = {-}2\tfrac{f''(r)}{f(r)} {-}4(m{-}1)\tfrac{h''(r)}{h(r)} {-}4(m{-}1)\tfrac{f'(r)}{f(r)}\tfrac{h'(r)}{h(r)} &{-}2(m{-}1)(2m{-}3)\tfrac{h'(r)^2}{h(r)^2} \\
&{+}\tfrac{4m(m{-}1)}{h(r)^2} {-}2(m-1)\big(\tfrac{mn}p\big)^2\tfrac{f(r)^2}{h(r)^4} \,\, .
\end{align*}
\end{prop}

We compute now the Chern connection and the Chern-Ricci tensors of the ma\-nifolds $(M_{(i,n)}(\fG,\fK),\bJ,\bg)$. From Proposition \ref{propCh} and Equations \eqref{eq:metric}, \eqref{Jad(T)} we get

\begin{prop} \label{prop:cher-connetion}
Let $X,Y,Z \in \gp$. Then
\begin{align*}
\n_{Y^*}X^*\big|_{\g(r)} &= \tfrac\l2\big(Q(JX,Y)T^*_{\g(r)} +Q(X,Y)N_{\g(r)}\big) \,\,, \\
\n_{T^*}X^*\big|_{\g(r)} &=\tfrac{2mn}{\l p}f(r)\tfrac{h'(r)}{h(r)}(JX)^*_{\g(r)} \,\,, \\
\n_{N}X^*\big|_{\g(r)} &=\tfrac{2mn}{\l p}f(r)\tfrac{h'(r)}{h(r)}X^*_{\g(r)} \,\,, \\
\n_{Y^*}T^*\big|_{\g(r)} &= \big(\l-\tfrac2{\l} \big(\tfrac{mn}p\big)^2\tfrac{f(r)^2}{h(r)^2}\big)(JY)^*_{\g(r)} \,\,, \\
\n_{T^*}T^*\big|_{\g(r)} &= -\tfrac{2mn}{\l p}f'(r)N_{\g(r)} \,\,, \\
\n_{N}T^*\big|_{\g(r)} &= \tfrac{2mn}{\l p}f'(r)T^*_{\g(r)} \,\,, \\
\n_{Y^*}N\big|_{\g(r)} &= -\tfrac2{\l} \big(\tfrac{mn}p\big)^2 \tfrac{f(r)^2}{h(r)^2}Y^*_{\g(r)} \,\,, \\
\n_{T^*}N\big|_{\g(r)} &= \tfrac{2mn}{\l p}f'(r)T^*_{\g(r)} \,\,, \\
\n_{N}N\big|_{\g(r)} &= \tfrac{2mn}{\l p}f'(r)N_{\g(r)} \,\,.
\end{align*}
\end{prop}

We are ready to state the following proposition, whose proof will be given in Appendix \ref{A}.

\begin{prop} \label{prop:chern-ricci}
Let $X \in \gp$ with $Q(X,X)=1$.
\begin{itemize}
\item[a)] The first Chern-Ricci tensor verifies
\begin{equation} \label{ric1} \begin{aligned}
\Ric^{\rm Ch[1]}(\bg)(T^*,T^*)_{\g(r)} &= \tfrac{2m(m-1)n^2}{p^2}f(r)^2\Big(-\tfrac{f''(r)}{f(r)} +(m{-}1)\Big(-\tfrac{h''(r)}{h(r)} +\tfrac{h'(r)^2}{h(r)^2} -\tfrac{f'(r)}{f(r)}\tfrac{h'(r)}{h(r)}\Big)\Big) \,\, , \\
\Ric^{\rm Ch[1]}(\bg)(X^*,X^*)_{\g(r)} &= h(r)^2\Big(\tfrac{2mn}p\tfrac{f(r)}{h(r)^2}\Big(\tfrac{f'(r)}{f(r)} +(m{-}1)\tfrac{h'(r)}{h(r)}\Big) +\tfrac{2m}{h(r)^2}\Big) \,\, .
\end{aligned} \end{equation}
\item[b)] The second Chern-Ricci tensor verifies
\begin{equation} \label{ric2} \begin{aligned}
\Ric^{\rm Ch[2]}(\bg)(T^*,T^*)_{\g(r)} &= \tfrac{2m(m-1)n^2}{p^2}f(r)^2 \Big(-\tfrac{f''(r)}{f(r)} +\tfrac{2m(m-1)n}{p}\tfrac{f'(r)}{h(r)^2} +\tfrac{2m^2(m-1)n^2}{p^2}\tfrac{f(r)^2}{h(r)^4}\Big) \,\, , \\
\Ric^{\rm Ch[2]}(\bg)(X^*,X^*)_{\g(r)} &= h(r)^2\Big({-}\tfrac{h''(r)}{h(r)}{+}\tfrac{h'(r)^2}{h(r)^2}{-}\tfrac{f'(r)}{f(r)}\tfrac{h'(r)}{h(r)}{+}\tfrac{2m(m-1)n}{p}f(r)\tfrac{h'(r)}{h(r)^3}{-}2\big(\tfrac{mn}{p}\big)^2\tfrac{f(r)^2}{h(r)^4}{+}\tfrac{2m}{h(r)^2}\Big) \,\, .
\end{aligned} \end{equation}
\item[c)] Both the Chern-Ricci tensors satisfy
\begin{equation} \label{ric12} \begin{gathered}
\Ric^{\rm Ch[i]}(\bg)(N,N)_{\g(r)} = \Ric^{\rm Ch[i]}(\bg)(T^*,T^*)_{\g(r)} \,\, , \\
\Ric^{\rm Ch[i]}(\bg)(N,T^*)_{\g(r)} = \Ric^{\rm Ch[i]}(\bg)(N,X^*)_{\g(r)} = \Ric^{\rm Ch[i]}(\bg)(T^*,X^*)_{\g(r)} =0 \,\, .
\end{gathered} \end{equation}
\item[d)] The Chern-scalar curvature is given by
\begin{multline} \label{scal}
\scal^{\rm Ch}(\bg)(r) = -2\tfrac{f''(r)}{f(r)} -2(m{-}1)\tfrac{h''(r)}{h(r)} +2(m{-}1)\Big(\tfrac{h'(r)}{h(r)} -\tfrac{f'(r)}{f(r)}\Big)\tfrac{h'(r)}{h(r)} +4m(m{-}1)\tfrac1{h(r)^2} \\
\phantom{=,} +\tfrac{4m(m-1)n}{p}\big(f'(r) +(m{-}1)f(r)\tfrac{h'(r)}{h(r)}\big) \tfrac1{h(r)^2}  \,\, .
\end{multline} \end{itemize} \end{prop}

Given Proposition \ref{prop:chern-ricci}, we are now able to study the second-Chern-Einstein equations and the constant Chern-scalar curvature equation for this special class of Hermitian cohomogeneity one manifolds. This will be done in Section \ref{sec:chern-einstein}. \smallskip

Finally, we recall that the torsion $\t$ of the Chern connection is given by \begin{equation} -2\bg(\t(A,B),C) = \diff\w(\bJ A,B,C)+\diff\w(A,\bJ B,C) \label{tau} \end{equation} and that its trace $\q(A) \= \Tr(\t(A,\cdot))$ is called {\it Lee form}. We recall that it satisfies $\diff\omega^{m-1}=\q\wedge\omega^{m-1}$, see \cite[page 500]{gauduchon-2}. \smallskip

From Proposition \ref{propdw} and Formulas \eqref{Jad(T)}, \eqref{eq:metric} we obtain
\begin{cor} \label{cordw}
Let $X,Y,Z \in \gp$. Then
$$
\diff\w(X^*,Y^*,N)_{\g(r)} = \tfrac{4mn}{\l p}f(r)\big(h(r)h'(r)+\tfrac{mn}pf(r)\big)\rho(X,Y) \,\, ,
$$
where $\rho(X,Y) \= Q_{\gp}(JX,Y)$ is the $\fG$-invariant K\"ahler-Einstein form on $P$, and
$$
\diff\w(X^*,Y^*,Z^*)_{\g(r)} = \diff\w(X^*,Y^*,T^*)_{\g(r)} = \diff\w(X^*,T^*,N)_{\g(r)} = 0 \,\, .
$$
\end{cor}

\noindent and, by consequence

\begin{prop} \label{propLee}
Let $X,Y \in \gp$. Then it holds $\t(N,T^*)_{\g(r)} = \t(X^*,Y^*)_{\g(r)} = 0$ and
\begin{align*}
\t(N,X^*)_{\g(r)} &= \tfrac{2mn}{\l p} \tfrac{f(r)}{h(r)^2}\big(h(r)h'(r) +\tfrac{mn}p f(r)\big)X^*_{\g(r)} \,\, ,\\
\t(T^*,X^*)_{\g(r)} &= \tfrac{2mn}{\l p} \tfrac{f(r)}{h(r)^2}\big(h(r)h'(r) +\tfrac{mn}p f(r)\big)(JX)^*_{\g(r)} \,\, .
\end{align*}
Moreover, the Lee form $\q$ satisfies
\begin{equation} \label{theta}
\q(N)_{\g(r)} = \tfrac{4m(m-1)n}{\l p}\tfrac{f(r)}{h(r)^2}\big(h(r)h'(r) +\tfrac{mn}p f(r)\big)
\end{equation}
and $\q(T^*)_{\g(r)} = \q(X^*)_{\g(r)} = 0$.
\end{prop}

As a direct consequence of Proposition \ref{propLee}, whose proof will be given in Appendix \ref{A}, we get the following

\begin{cor} For any $A,B \in \G(TM)$, it holds that
\begin{equation} \label{eq:lck}
2(m{-}1)\t(A,B) = \big(\q(A)B-\q(B)A\big) -\big(\q(\bJ A)\bJ B-\q(\bJ B)\bJ A\big) \,\, . \end{equation}
\end{cor}

\section{Special Hermitian metrics on B\'erard-Bergery manifolds} \label{sec:metrics}
\setcounter{equation} 0

In this section, we investigate the existence of {\em special} non-K\"ahler Hermitian metrics, such as balanced, pluriclosed, locally conformally K\"ahler, Vaisman, and Gauduchon, on the B\'erard-Bergery standard cohomogeneity one Hermitian manifolds. In particular, we prove Theorem \ref{thm:a} and Theorem \ref{thm:b}.

\subsection{Proof of Theorem \ref{thm:a}}\label{sec:proof-thm-a} \hfill \par

We begin by pointing out the following

\begin{prop}\label{prop:LCK}
All the cohomogeneity one Hermitian metrics $\bg$ of submersion-type on the complex manifold $(M_{(i,n)}(\fG,\fK),\bJ)$ are locally conformally K\"ahler.
\end{prop}

\begin{proof}
By \cite[Corollary 1.1]{dragomir-ornea}, we know that $\bg$ is locally conformally K\"ahler if and only if the complex structure $\bJ$ is parallel with respect to the Weyl connection associated to $(\bg,\frac{1}{m-1}\q)$, equivalently, the following is satisfied:
$$
D_{A}\bJ B - \bJ D_A B = \tfrac{1}{2(m-1)} \big(\q(\bJ B)A - \q(B)\bJ A + \bg(A,B)\bJ\q^\# + \bg(\bJ A,B)\q^\# \big)
$$
for any $A, B \in \Gamma(TM)$, where $\q^\# \in \Gamma(TM)$ is defined by the relation $\bg(\q^\#,\cdot) = \q$.
By using Equations \eqref{defChern} and \eqref{tau}, a straightforward computation shows that the above equation is equivalent to Equation \eqref{eq:lck}. Indeed, for any $C\in\Gamma(TM)$,
$$\begin{aligned}
\bg(D_{A}\bJ B &- \bJ D_A B, C) \\
&= \bg(\n_{A}\bJ B,C) +\bg(\n_{A}B,\bJ C) + \tfrac12\diff\w(\bJ A, \bJ B, C) + \tfrac12\diff\w(\bJ A, B, \bJ C) \\
&= -\bg(\t(B,C),\bJ A) \\
&= \tfrac1{2(m-1)}\bg\big(\bJ A,(\q(\bJ B)\bJ C-\q(\bJ C)\bJ B)-(\q(B)C-\q(C)B)\big) \\
&= \tfrac1{2(m-1)}\bg\big(\q(\bJ B)\bg(A,C) -\q(B)\bg(\bJ A,C) +\bg(A,B)\bg(\bJ\q^\#,C) +\bg(\bJ A,B)\bg(\q^\#,C)\big) \\
&= \tfrac1{2(m-1)}\bg\big(\q(\bJ B)A - \q(B)\bJ A + \bg(A,B)\bJ\q^\# + \bg(\bJ A,B)\q^\#,C\big) ,
\end{aligned}$$
which shows the above mentioned equivalence.
\end{proof}

\begin{cor} \label{globCK}
Let $(M_{(i,n)}(\fG,\fK),\bJ,\bg)$ be a B\'erard-Bergery standard cohomogeneity one Hermitian ma\-ni\-fold. Then, $\bg$ is strictly locally conformally K\"ahler if and only if $i=3$.\end{cor}

\begin{proof}
Since the complex manifolds $(M_{(2,n)}(\fG,\fK), \bJ)$ and $(M_{(4,n)}(\fG,\fK), \bJ)$ are simply connected, any closed $1$-form on them is necessarily exact. Moreover, by Remark \ref{rem:noK}, it holds that $b_1\big(M_{(1,n)}(\fG,\fK)\big)=0$. Therefore, any locally conformally K\"ahler metrics on them are necessarily globally conformally K\"ahler. Finally, by using again Remark \ref{rem:noK}, the complex manifolds $(M_{(3,n)}(\fG,\fK), \bJ)$ do not admit any K\"ahler metric.
\end{proof}

Let now $\bg$ be a cohomogeneity one, submersion-type Hermitian metric on a complex manifold $(M_{(i,n)}(\fG,\fK),\bJ)$. By Proposition \ref{prop:LCK}, it follows that $\diff\omega=\tfrac{1}{m-1}\q\wedge\omega$ which in turn implies that $\eL_A\q=0$ for any holomorphic Killing vector field $A \in \Gamma(TM)$.
Hence, by Equations \eqref{defN}, \eqref{eq:metric} and Proposition \ref{prop:lc-connetion}, the non-vanishing components of the Levi-Civita covariant derivative of $\q$ are
\begin{equation} \label{LCtheta} \begin{aligned}
(D_N\q)(N)_{\gamma(r)} &= \tfrac{2mn}{\l p}\big(f(r) \tfrac{\p}{\p r}(\q(N)_{\gamma(r)})-f'(r)\q(N)_{\gamma(r)}\big) \,\, , \\
(D_{T^*}\q)(T^*)_{\gamma(r)} &= \tfrac{2mn}{\l p}f'(r)\q(N)_{\gamma(r)} \,\, , \\
(D_{Y^*}\q)(X^*)_{\gamma(r)} &= \tfrac{\l p}{2mn} \tfrac{h(r)h'(r)}{f(r)}Q(X,Y)\q(N)_{\gamma(r)} \,\, ,
\end{aligned} \end{equation}
where $\lambda$ is given by the Equation \eqref{lambda}. By Corollary \ref{globCK}, the complex manifolds $(M_{(1,n)}(\fG,\fK),\bJ)$, $(M_{(2,n)}(\fG,\fK),\bJ)$ and $(M_{(4,n)}(\fG,\fK),\bJ)$ cannot admit cohomogeneity one, Hermitian metric of submersion-type that are Vaisman. Moreover, from Equation \eqref{LCtheta}, we get

\begin{prop}\label{prop:vaisman}
A cohomogeneity one Hermitian metric $\bg$ of submersion-type on the complex manifolds $(M_{(3,n)}(\fG,\fK),\bJ)$ is Vaisman if and only if both $f$, $h$ are constant.
\end{prop}

\begin{proof}
The {\itshape if part} is immediate. Indeed, since $\q$ is necessarily non-exact, if both $f$ and $h$ are constant, then by Equations \eqref{theta} and \eqref{LCtheta} it follows that $D\q=0$.

Conversely, assume that $D\q=0$. Notice that either $\q(N)_{\g(r_{\zero})}=0$ for some $r_{\zero}\in I$, or $\q(N)_{\gamma(r)}$ is nowhere vanishing. In the former case, the first equation in \eqref{LCtheta} yields that $\q(N)_{\g(r)}$ is constantly zero. In fact, one gets that $\q=0$, that is, $\bg$ is K\"ahler. In particular, if $\bg$ is Vaisman, then the above observation implies that $D\q=0$ and $\q(N)_{\g(r)}$ is nowhere vanishing. Hence, Equations \eqref{LCtheta} immediately imply that $f$ and $h$ are constant. 
\end{proof}

Finally, we note that the manifolds $(M_{(1,n)}(\fG,\fK),\bJ)$ (respectively $(M_{(3,n)}(\fG,\fK),\bJ)$) are acted transitively by the larger group $\fG\times\bR$ (respectively $\fG\times\fU(1)$), and any metric $\bg=\bg(f,h)$ is invariant under this action if and only if the functions $f$ and $h$ are constant. This completes the proof of Theorem \ref{thm:a}.

\begin{remark}
By Proposition \ref{prop:LCK} and Proposition \ref{prop:vaisman}, the compact complex manifolds $(M_{(3,n)}(\fG,\fK),\bJ)$ admit cohomogeneity one, strictly locally conformally K\"ahler metrics that are non-Vaisman. Remarkably, in the homogeneous case, this is excluded by \cite{gauduchon-moroianu-ornea, hasegawa-kamishima}. We stress that the Hopf manifold, which is the main example of Vaisman manifold \cite{ornea-verbitsky-1}, corresponds, in our notation, to $(M_{(3,1)}(\fG,\fK),\bJ)$ with $\fG=\fSU(m)$ and $\fK=\fS(\fU(1)\times\fU(m{-}1))$, see Example \ref{ex:hopf}.
\end{remark}

\begin{remark}
The locally conformally K\"ahler metrics by Proposition \ref{prop:LCK} include the globally conformally K\"ahler, Einstein metrics by B\'erard-Bergery \cite[Th\'eor\`eme 1.10]{berard-bergery}. We also recall that Einstein, locally conformally K\"aher, non-K\"ahler metrics are completed classified by \cite{lebrun, derdzinski-maschler, madani-moroianu-pilca}, and they are either the Einstein, globally conformally K\"ahler metrics by \cite{berard-bergery}, or they are defined on $\bC \bP^2$ by blowing up one or two points.
\end{remark}

\subsection{Proof of Theorem \ref{thm:b}}\label{sec:proof-thm-b} \hfill \par

We begin by characterizing the Gauduchon condition as follows.

\begin{prop} \label{prop:gauduchon}
A cohomogeneity one Hermitian metric $\bg$ of submersion-type on the complex manifolds $(M_{(i,n)}(\fG,\fK),\bJ)$ is Gauduchon if and only if it satisfies
\begin{equation} \label{eq:gauduchon}
\big(h(r)h'(r)+\tfrac{mn}{p} f(r)\big)f(r)h(r)^{2(m-2)}=k \quad \text{ for some $k \in \bR$.}
\end{equation}
Moreover, if $(M_{(i,n)}(\fG,\fK),\bJ)$ has singular orbits, then $\bg$ is Gauduchon if and only if it is K\"ahler.
\end{prop}

\begin{proof} Let $(\tilde{e}_{\a})$ be a $Q_{\gp}$-orthonormal basis for $\gp$. Then, a straightforward computation based on Equation \eqref{LCtheta} yields
\begin{align*}
\diff{}^*\q(r) &= -\big(\tfrac{\l p}{2mn}\big)^2f(r)^{-2}(D_N\q)(N)_{\g(r)} -\big(\tfrac{\l p}{2mn}\big)^2f(r)^{-2}(D_{T^*}\q)(T^*)_{\g(r)} -h(r)^{-2}\sum_{\a=1}^{2(m-1)}(D_{\tilde{e}_{\a}^*}\q)(\tilde{e}_{\a}^*)_{\g(r)} \\
&= -\tfrac{\l p}{2mn}f(r)^{-1} \tfrac{\p}{\p r}(\q(N)_{\gamma(r)}) -(m-1)\tfrac{\l p}{mn}\tfrac{h'(r)}{f(r)h(r)}\q(N)_{\gamma(r)}\\
&= 2(m-1)\tfrac1{h(r)^2}\Big(\big(h(r)h'(r)+\tfrac{mn}p f(r)\big)\Big(\tfrac{f'(r)}{f(r)}+2(m-2)\tfrac{h'(r)}{h(r)}\Big) +\big(h(r)h''(r)+h'(r)^2+\tfrac{mn}p f'(r)\big)\Big)
\end{align*} and so $\bg$ is Gauduchon if and only if
$$
\big(h(r)h'(r)+\tfrac{mn}{p} f(r)\big)\Big(\tfrac{f'(r)}{f(r)}+2(m-2)\tfrac{h'(r)}{h(r)}\Big) +\big(h(r)h''(r)+h'(r)^2+\tfrac{mn}{p} f'(r)\big)=0 \,\, .
$$
Since $f(r), h(r)$ are positive for any $r \in I$ and
\begin{multline*}
\tfrac1{f(r)h(r)^{2(m-2)}}\tfrac{\diff}{\diff r}\big(\big(h(r)h'(r)+\tfrac{mn}{p} f(r)\big)f(r)h(r)^{2(m-2)}\big)  = \\
\quad \big(h(r)h'(r)+\tfrac{mn}{p} f(r)\big)\Big(\tfrac{f'(r)}{f(r)}+2(m-2)\tfrac{h'(r)}{h(r)}\Big) +\big(h(r)h''(r)+h'(r)^2+\tfrac{mn}{p} f'(r)\big) \,\, ,
\end{multline*}
it follows that $\bg$ is Gauduchon if and only if Equation \eqref{eq:gauduchon} is satisfied.

Let us assume now that $(M_{(i,n)}(\fG,\fK),\bJ)$ has a singular orbit, that is $i=2$ or $i=4$. Then, the smoothness conditions at $r=0$ imply that $k=0$ in Equation \eqref{eq:gauduchon}. Therefore, in this case, $\bg$ is Gauduchon if and only if it is K\"ahler.
\end{proof}

Since the balanced condition is equivalent to $\q=0$, from Proposition \ref{propLee} and Equation \eqref{eq:kahler} we immediately get

\begin{cor}\label{cor:balanced}
A cohomogeneity one Hermitian metric $\bg$ of submersion-type on the complex manifolds $(M_{(i,n)}(\fG,\fK),\bJ)$ is balanced if and only if it is K\"ahler.
\end{cor}

\begin{remark}
In particular, in view of \cite[Theorem 1.3]{yang-zheng}, in the non-K\"ahler case, both the Levi-Civita and the Chern connections cannot be {\em K\"ahler-like} in the sense of \cite{gray, yang-zheng}, namely, they do not satisfy the same symmetries as in the K\"ahler case.
\end{remark}

Concerning the pluriclosed condition, a tedious but straightforward computation (see Appendix \ref{A}) shows that 
\begin{equation} \label{ddcw}
\diff\diff^{\,c}\!\w(X^*,Y^*,Z^*,W^*)_{\g(r)} = 4\tfrac{mn}{p}f(r)\big(h(r)h'(r)+ \tfrac{mn}{p}f(r)\big)(\rho{\wedge}\rho)(X,Y,Z,W) \,\, ,
\end{equation}
where $\rho(X,Y) = Q_{\gp}(JX,Y)$ is again the $\fG$-invariant K\"ahler-Einstein form on $P$. Hence, together with Equation \eqref{eq:kahler}, this proves the following 

\begin{cor}\label{cor:skt}
A cohomogeneity one Hermitian metric $\bg$ of submersion-type on the complex manifolds $(M_{(i,n)}(\fG,\fK),\bJ)$ is pluriclosed if and only if it is K\"ahler.
\end{cor}

\noindent which completes the proof of Theorem \ref{thm:b}.

\section{Constant Chern-scalar curvature and second-Chern-Einstein metrics} \label{sec:chern-einstein}
\setcounter{equation} 0

In this section, we investigate the existence of second-Chern-Einstein metrics and of metrics with constant Chern-scalar curvature on the B\'erard-Bergery standard cohomogeneity one Hermitian manifolds. In particular, we first prove a local existence and uniqueness result for second-Chern-Einstein metrics with prescribed Chern-scalar curvature. Then, we prove Theorem \ref{thm:c} and Theorem \ref{thm:d}.

\subsection{The second-Chern-Einstein equations} \hfill \par

Let $(M_{(i,n)}(\fG,\fK),\bJ,\bg)$ be a B\'erard-Bergery standard cohomogeneity one Hermitian manifold and fix a unit speed geodesic $\g:\ol{I} \to M$ which intersects orthogonally any $\fG$-orbit. Then, by means of Proposition \ref{prop:chern-ricci}, the second-Chern-Einstein equation
$$
\Ric^{{\rm Ch}[2]}(\bg)=\tfrac{\l}{2m}\bg
$$
becomes
\begin{equation} \label{IICE} \left\{\begin{array}{lcl}
-\tfrac{f''(r)}{f(r)} +\tfrac{2m(m-1)n}{p}\tfrac{f'(r)}{h(r)^2} +\tfrac{2m^2(m-1)n^2}{p^2}\tfrac{f(r)^2}{h(r)^4} \!\!&=&\!\! \tfrac{\l(r)}{2m} \\
-\tfrac{h''(r)}{h(r)} +\tfrac{h'(r)^2}{h(r)^2} -\tfrac{f'(r)}{f(r)}\tfrac{h'(r)}{h(r)} +\tfrac{2m(m-1)n}{p}f(r)\tfrac{h'(r)}{h(r)^3} -2\big(\tfrac{mn}{p}\big)^2\tfrac{f(r)^2}{h(r)^4} +\tfrac{2m}{h(r)^2} \!\!&=&\!\! \tfrac{\l(r)}{2m}
\end{array}\right. \,\, ,\end{equation}
where $\lambda: I \to \bR$ is an admissible function (see Remark \ref{admissible}). Notice that, in this case, it holds that $\l=\scal^{\rm Ch}(\bg)$. \smallskip

Our first result in this section concerns the local existence and uniqueness of second-Chern-Einstein metrics, with prescribed Chern-scalar curvature, in a neighborhood of a singular orbit. More precisely

\begin{theorem} \label{thm:loc}
Assume that the complex manifold $(M_{(i,n)}(\fG,\fK),\bJ)$ has a singular orbit, corresponding to the value $r=0$ of the orthogonal geodesic $\g$. For any constant $a \in \bR_{>0}$ and for any admissible function $\l: I \to \bR$, there exist $\e>0$ and two smooth functions $f,h : [0,\e) \to \bR$ satisfying the following conditions: \begin{itemize}
\item[I)] $f,h$ solve the second-Chern-Einstein Equations \eqref{IICE};
\item[II)] $f,h$ determine a smooth Hermitian metric $\bg=\bg(f,h)$ on the open set
$$
\cU^{\rm reg}_{\e} \= \bigcup_{r \in (0,\e)} \fG \cdot \g(r) \subset M_{(i,n)}(\fG,\fK)^{\rm reg}
$$
which extends smoothly over the singular orbit $\fG \cdot \g(0)$;
\item[III)] $h(0)=a$ and the Chern-scalar curvature of $\bg$ verifies $\scal^{\rm Ch}(\bg)(r)=\l(r)$ for any $r \in [0,\e)$.
\end{itemize}
Moreover, it holds that: \begin{itemize}
\item[$\bcdot$] if there exist $\tilde{\e}\geq \e$ and $\tilde{f},\tilde{h} : [0,\tilde{\e}) \to \bR$ satisfying the conditions (I), (II), (III) above, then $\tilde{f}(r)=f(r)$ and $\tilde{h}(r)=h(r)$ for any $r \in [0,\e)$;
\item[$\bcdot$] $f,h$ depends continuously on the data $a, \l$. 
\end{itemize}
\end{theorem}

\begin{proof}
Fix a positive number $a>0$ and an admissible function $\l: I \to \bR$. Let us write
\begin{equation} \label{xtof}
x(r) \= \tfrac{f(r)}{r} \,\, , \quad y(r)\= x'(r) \,\, , \quad z(r) \= h(r) \,\, , \quad w(r) \= h'(r) \,\, .
\end{equation}
Then, a straightforward computation shows that the Equations \eqref{IICE} become
\begin{equation} \label{IICEMal}
\left\{ \begin{array}{l}
v'(r) = \tfrac1r\,A\cdot v(r)+ N(r,v(r)) \\
v(0) = v_{\zero}
\end{array}\right.  \,\, , \end{equation}
with
$$\begin{gathered}
v(r) \= (x(r), y(r), z(r), w(r))^t \,\, , \quad A \= \left(\begin{array}{cccc}
0 & 0 & 0 & 0 \\
0 & -2 & 0 & 0 \\
0 & 0 & 0 & 0 \\
0 & 0 & 0 & -1
\end{array}\right) \,\, , \\
N(r,v) \= \left(\begin{array}{c}
y \\
\big(\tfrac{2m(m-1)n}{p}x^2 -\tfrac{\l}{2m}x\big) +r\big(\tfrac{2m(m-1)n}{p}xy\big) +r^2\big(\tfrac{2m^2(m-1)n^2}{p^2}\tfrac{x^3}{z^4}\big) \\
w \\
\big(\tfrac{w^2}{z} -\tfrac{yw}{x} +\tfrac{2m}{z} -\tfrac{\l}{2m}z\big) +r\big(\tfrac{2m(m-1)n}{p}\tfrac{xw}{z^2}\big) +r^2\big(-\tfrac{2m^2(m-1)n^2}{p^2}\tfrac{x^2}{z^3}\big)
\end{array}\right) \,\, .
\end{gathered}$$
Moreover, the smoothness conditions for the functions $f$ and $h$, together with the equation $h(0)=a$, imply that
\begin{equation} \label{icMal}
v_{\zero} = (1, 0, a, 0)^t \,\, .
\end{equation}
We stress now that the following conditions are satisfied: \begin{itemize}
\item[$\bcdot$] the function $N=N(r,v)$ is smooth in a neighborhood of $(0,v_{\zero})$,
\item[$\bcdot$] $A \cdot v_{\zero} = 0$,
\item[$\bcdot$] $\det(A -k \Id_4) \neq 0$ for any integer $k\geq1$.
\end{itemize}
Then, by the Malgrange Theorem \cite[Theorem 7.1]{malgrange}, see also \cite[Theorem 2.2]{boehm98}, there exists a unique solution $v(r)$, defined on an interval $(-\e,\e)$, to the Equation \eqref{IICEMal} with initial condition \eqref{icMal}, which depends continuously on the data $a$, $\l$.

By Equation \eqref{xtof}, we obtain a pair $(f,h)$ of smooth functions $f,h : (-\e,\e) \to \bR$ which satisfy Equations \eqref{IICE} such that
\begin{equation} \label{CI-IICE}
f(0)=0 \,\, , \quad f'(0)=1 \,\, , \quad f''(0)=0 \,\, , \quad h(0)=a \,\, , \quad h'(0)=0 \,\, .
\end{equation}
Since the pair $(\hat{f},\hat{h})$ of functions defined by
$$
\hat{f},\hat{h} : (-\e,\e) \to \bR \,\, , \quad \hat{f}(r) \= -f(-r) \,\, , \quad \hat{h}(r) \= h(-r)
$$
satisfy Equations \eqref{IICE} with the initial conditions \eqref{CI-IICE}, by uniqueness we conclude that $f$ is odd and $h$ is even. Therefore, these functions give rise to a smooth Hermitian metric $\bg=\bg(f,h)$ on the open set $\cU^{\rm reg}_{\e} \subset M_{(i,n)}(\fG,\fK)^{\rm reg}$, which admits a unique smooth extension over the singular orbit $\fG \cdot \g(0)$. 
\end{proof}

Concerning complete solutions to the Equations \ref{IICE}, we point out the following

\begin{rem} \label{rem:homCE}
Fundamental examples of complete, non-K\"ahler, second-Chern-Einstein metrics can be easily found on the manifolds $(M_{(1,n)}(\fG,\fK),\bJ)$ and $(M_{(3,n)}(\fG,\fK),\bJ)$. Indeed, the constant functions
\begin{equation} \label{hom13}
f(r) \= \tfrac{p}{mn} \,\, , \quad h(r) \= 1
\end{equation}
verify the smoothness conditions in case $i=1,3$ and so they give rise to homogeneous, smooth metrics which are non-K\"ahler by Equation \eqref{eq:kahler} and second-Chern-Einstein by Equation \eqref{IICE} with constant Chern-scalar curvature $\l=4m(m-1)$. These examples include the standard metric on the linear Hopf manifold (see Example \ref{ex:hopf}).
\end{rem}

We also stress that, on manifolds $(M_{(1,n)}(\fG,\fK),\bJ)$ and $(M_{(3,n)}(\fG,\fK),\bJ)$, all the metrics (not necessarily of cohomogeneity one) in the conformal class of $\bg=\bg(f,h)$, with $f,h$ given by Formula \eqref{hom13}, are second-Chern-Einstein (see Remark \ref{rmk:conf-invariant}). In particular, by Remark \ref{confchange}, for any admissible positive function $\phi: I \to \bR$, the pair
\begin{equation} \label{confPHI}
f_{\phi}(r) \= \tfrac{p}{mn}\phi(r) \,\, , \quad h_{\phi}(r) \= \phi(r)
\end{equation}
solve the Equations \eqref{IICE} with Chern-scalar curvature $\l(r) = 2m\big(-\tfrac{\phi''(r)}{\phi(r)}+2(m-1)\tfrac{\phi'(r)+1}{\phi(r)}\big)$. Therefore, in the following, we will focus on B\'erard-Bergery manifolds $(M_{(i,n)}(\fG,\fK),\bJ)$ with singular orbits, namely, the cases $i=2$ and $i=4$.

\subsection{Complete second-Chern-Einstein metrics in case of singular orbits} \hfill \par

In this section, we will construct complete second-Chern-Einstein metrics on the manifolds $(M_{(2,n)}(\fG,\fK),\bJ)$ by using the same technique as \cite[Section 11]{berard-bergery}.

Let us start by noticing that, for the manifold $\cO_{\bC\bP^{m-1}}(-1)$, the functions
\begin{equation} \label{fhktaut}
f(r) \= r \,\, , \quad h_k(r) \= \sqrt{r^2+k^2} \,\, , \quad r \in [0,+\infty) \,\, , \quad k>0
\end{equation}
solve the second-Chern-Einstein Equations \eqref{IICE} and define Hermitian metrics $\bg_k=\bg(f,h_k)$ which extend smoothly over the singular orbit. All the metrics $\bg_k$ are non-K\"ahler and their Chern-scalar curvatures are given by
$$
\scal^{\rm Ch}(\bg_k)(r) = 4m(m-1)\tfrac{2r^2+k^2}{(r^2+k^2)^2} \,\, .
$$
Notice that all these metrics are homothetic, indeed the satisfy $\bg_k = k^2\bg_1$ (see Remark \ref{confchange}). More in general, we have

\begin{theorem} \label{thm:IICE2}
All the complex manifolds $(M_{(2,n)}(\fG,\fK),\bJ)$ have a complete, Hermitian, non-K\"ahler, second-Chern-Einstein, cohomogeneity one metric.
\end{theorem}

\begin{proof}
Fix a manifold $(M_{(2,n)}(\fG,\fK),\bJ)$ and set
\begin{equation} \label{ansM2}
f_{\phi}(r) \= \tfrac{p}{2mn}\phi'(r) \,\, , \quad h_{\phi}(r) \= \sqrt{\phi(r)}
\end{equation}
for some smooth, positive, increasing function $\phi: [0,+\infty) \to \bR$. Notice that, in this case
$$
h_{\phi}(r)h_{\phi}'(r)+\tfrac{mn}{p}f_{\phi}(r) = \phi'(r)
$$
and so this metric is necessarily non-K\"ahler by Equation \eqref{eq:kahler}. Setting the initial condition $h(0)=1$, the second-Chern-Einstein Equations \eqref{IICE} become
\begin{equation} \label{anseqM2}
\left\{ \begin{array}{l}
\phi(r)\tfrac{\phi'''(r)}{\phi'(r)} -m\phi''(r)+2m=0 \\
\phi(0)=1 \,\, , \quad \phi'(0)=0 \,\, , \quad \phi''(0)=\tfrac{2mn}{p}
\end{array}\right.  \,\, . \end{equation}
The Cauchy problem \eqref{anseqM2} admits a unique smooth solution on some interval $[0,\e)$, which extends to an even smooth function on $(-\e,\e)$. Let us prove that this solution can be extended to the whole $[0,+\infty)$.

Assume that the solution to Equation \eqref{anseqM2} is of the form
\begin{equation} \label{anseq'M2}
\phi'(r)=\sqrt{u(\phi(r))} \,\, .
\end{equation}
Then, from Equations \eqref{anseqM2} and \eqref{anseq'M2}, we get the following Cauchy problem for $u(t)$:
\begin{equation} \label{anseq''M2}
\left\{ \begin{array}{l}
tu''(t) -mu'(t) +4m=0 \\
u(1)=0 \,\, , \quad u'(1)=\tfrac{4mn}{p}
\end{array}\right.  \,\, . \end{equation}
The unique solution to Equation \eqref{anseq''M2} is the function $u: [1,+\infty) \to \bR$ defined by
\begin{equation} \label{uM2}
u(t) \= -\tfrac{4m(n+p)}{p(m+1)}+4t+\tfrac{4(mn-p)}{p(m+1)}t^{m+1} \,\, ,
\end{equation}
which is smooth, positive and increasing. Hence, the function
$$
\f: [1,+\infty) \to \bR \,\, , \quad \f(r) \= \int_1^r\frac{\diff t}{\sqrt{u(t)}}
$$
is smooth, positive, increasing and, by construction, its inverse $\phi \= \f^{-1}$ solves the Cauchy problem \eqref{anseqM2}. Therefore, by means of Equation \eqref{ansM2}, the proof is completed. \end{proof}

\begin{rem}
The Chern-scalar curvature of the metric $\bg_{\phi}=\bg(f_{\phi},h_{\phi})$ constructed from Equation \eqref{ansM2} by solving the Cauchy problem \eqref{anseqM2} is given by
$$
\scal^{\rm Ch}(\bg_{\phi})(r) = 2m\Big(-\tfrac{\phi'''(r)}{\phi'(r)} +(m-1)\tfrac{\phi''(r)}{\phi(r)} +(m-1)\big(\tfrac{\phi'(r)}{\phi(r)}\big)^2\Big) \,\, .
$$
Notice that, if $mn-p=0$, which correspond to the manifolds $\cO_{\bC\bP^{m-1}}(-1)$, the function \eqref{uM2} is $u(t)=4(t-1)$. Hence, we recover the family of examples introduced in Formula \eqref{fhktaut}.
\end{rem}

Finally we observe that, concerning the compact simply-connected manifolds $(M_{(4,n)}(\fG,\fK),\bJ)$, we have the following

\begin{prop} \label{noCE4}
On the complex manifolds $(M_{(4,n)}(\fG,\fK),\bJ)$ there are no cohomogenity one, submersion-type Hermitian metrics that are second-Chern-Einstein.
\end{prop}

\begin{proof} Assume that there exists cohomogenity one, submersion-type Hermitian metric $\bg$ on a complex manifold $(M_{(4,n)}(\fG,\fK),\bJ)$ which is second-Chern-Einstein. Then, by Corollary \ref{globCK}, it is globally conformally K\"ahler. By Remark \ref{rmk:conf-invariant}, this implies the existence of a K\"ahler-Einstein metric on $(M_{(4,n)}(\fG,\fK),\bJ)$, that is not possible (see \cite[Remarques 8.14, (1)]{berard-bergery} and \cite[Remarks 9.126, (b)]{besse}).
\end{proof}

\begin{rem} \label{rem:sconfitta}
The projective space $\bC\bP^m$ is a standard cohomogeneity one manifold with respect to the action of $\fG= \fSU(m)$ given by $a \cdot [z^0 : z] \= [z^0 : a \cdot z] \,\, .$ Even if it is not a Berard-B\'erg\'ery manifold according to the Definition \ref{def:bb}, all the formulas in Section \ref{sec:bb} and Section \ref{sec:metrics} still apply to this specific case. In particular, all the $\fG$-invariant metrics on $\bC\bP^m$ are of the form \eqref{eq:metric} with $n=1$ and $p=m$, where $f,h: [0,\tfrac{\pi}2] \to \bR$ are smooth, positive function satisfying: \begin{itemize}
\item[$\bcdot$] $f$ is the restriction of a smooth odd function on $\bR$ satisfying
$$
f(r+\tfrac{\pi}2) = -f(r-\tfrac{\pi}2) \quad \text{ and } \quad f'(0)=1=-f'(\tfrac{\pi}2) \,\, ;
$$
\item[$\bcdot$] $h$ is the restriction of a smooth even function on $\bR$ satisfying
$$
h(r+\tfrac{\pi}2) = -h(r-\tfrac{\pi}2) \quad \text{ and } \quad h'(\tfrac{\pi}2)=-1 \,\, .
$$
\end{itemize}
Notice that the Fubini-Study metric $\bg_{\rm FS}$ with sectional curvature $1 \leq \sec \leq 4$ corresponds to the functions
$$
f(r)\= \tfrac12\sin(2r) \,\, , \quad h(r)\= \cos(r) \,\, .
$$
As argued in the proof of Proposition \ref{noCE4}, all the $\fG$-invariant second-Chern-Einstein metrics on $\bC\bP^m$ are necessarily conformal to the Fubini-Study metric $\bg_{\rm FS}$. For example, the functions
$$
f(r)\= \tfrac12\sin(2r) \,\, , \quad h_k(r)\= \cos(r)\sqrt{\sin(r)^2+k^2\cos(r)^2} \,\, , \quad k>0
$$
define  non-K\"ahler, second-Chern-Einstein metrics $\bg_k$ on $\bC\bP^m$ of the form
$$
\bg_k = \f_k^2\, \bg_{\rm FS} \,\, , \quad \text{ with } \quad \f_k(r) \= \tfrac{k}{\cos(r)^2+k^2\sin(r)^2} \,\, .
$$
\end{rem}

\subsection{Constant Chern-scalar curvature metrics in case of singular orbits} \hfill \par

In this section, we construct complete constant Chern-scalar curvature metrics $\bg=\bg(f,h)$ on the complex manifolds $(M_{(2,n)}(\fG,\fK),\bJ)$ and $(M_{(4,n)}(\fG,\fK),\bJ)$ by using again the technique exploited in \cite[Section 11]{berard-bergery}.

Fix a complex manifold $(M_{(i,n)}(\fG,\fK),\bJ)$, with $i\in\{2,4\}$, and set
\begin{equation} \label{ansCh}
f_{\phi}(r) \= \tfrac{p}{2mn}\phi(r)\phi'(r) \,\, , \quad h_{\phi}(r) \= \phi(r)
\end{equation}
for some smooth, increasing, positive, function $\phi: I \to \bR$. Notice that, in this case,
$$
h_{\phi}(r)h_{\phi}'(r)+\tfrac{mn}{p}f_{\phi}(r) = \tfrac12\phi(r)\phi'(r)
$$
and so this metric is necessarily non-K\"ahler by Equation \eqref{eq:kahler}. Let $c \in \bR$ to be fixed later. Then, the constant Chern-scalar curvature equation
$$
\scal^{\rm Ch}(\bg_{\phi})=c
$$
for the metric $\bg_{\phi} \= \bg(f_{\phi},h_{\phi})$ becomes
\begin{equation} \label{anseqscal}
\phi(r)^2\tfrac{\phi'''(r)}{\phi'(r)} +(m+2)\phi(r)\phi''(r) -m(m-1)\phi'(r)^2 +\tfrac{c}2\phi(r)^2-2m(m-1)=0 \,\, .
\end{equation}
We look for a solution of the form
\begin{equation} \label{anseqscal'}
\phi'(r)=\sqrt{u(\phi(r))}
\end{equation}
for some smooth real function $u=u(t)$. The, we get the following ODE
\begin{equation} \label{ODEu(t)}
t^2u''(t) +(m+2)tu'(t) -2m(m-1)u(t) +ct^2-4m(m-1)=0 \,\, ,
\end{equation}
which can be explicitly integrated. Indeed, the following cases occur. \begin{itemize}

\item[$\bcdot$] If $m=2$, then the solutions to Equation \eqref{ODEu(t)} are
\begin{equation} \label{solm=2}
u_{a,b,c}(t) = at^{-4} -2 +bt-\tfrac{c}6t^2 \,\, , \quad \text{ with } a,b \in \bR \,\, .
\end{equation}
In this case, the base space of the fibration $M_{(i,n)}(\fG,\fK) \to P$ is necessarily $P=\bC \bP^1$ and so $p=2$. 

\item[$\bcdot$] If $m=3$, then the solutions to Equation \eqref{ODEu(t)} are
\begin{equation} \label{solm=3}
u_{a,b,c}(t) = at^{-6} -2 +bt^2-\tfrac{c}8\log(t)t^2 \,\, , \quad \text{ with } a,b \in \bR \,\, .
\end{equation}
In this case, the only possibilities for the base space of the fibration $M_{(i,n)}(\fG,\fK) \to P$ are
$$
P=\bC \bP^2 = \fSU(3)/\fS(\fU(1){\times}\fU(2)) \,\, , \quad P={\rm Gr}(2,\bR^5) = \fSp(2)/\fU(2)
$$
and so $p=3$. 

\item[$\bcdot$] If $m>3$, then the solutions to Equation \eqref{ODEu(t)} are given by
\begin{equation} \label{solm>3}
u_{a,b,c}(t) = at^{-2m} -2 +\tfrac{c}{2(m+1)(m-3)}t^2 +bt^{m-1} \,\, , \quad \text{ with } a,b \in \bR \,\, .
\end{equation}
\end{itemize}

Then, by means of Equations \eqref{ansCh} and \eqref{anseqscal'}, we are able to construct constant Chern-scalar curvature metrics on the manifolds $(M_{(2,n)}(\fG,\fK),\bJ)$ and $(M_{(4,n)}(\fG,\fK),\bJ)$.

\begin{theorem}\label{thm:chcsc-M2}
Let $c \in \bR$, $c\leq0$. Then, all the complex manifolds $(M_{(2,n)}(\fG,\fK),\bJ)$ have a complete, Hermitian, non-K\"ahler, cohomogeneity one metric $\bg$ with $\scal^{\rm Ch}(\bg)=c$.
\end{theorem}

\begin{proof}
Setting $h(0)=1$, the smoothness conditions for $f,h$ imply that
\begin{equation} \label{icM2} \begin{gathered}
\phi(0)=1 \,\, , \quad \phi'(0)=0 \,\, , \quad \phi''(0)=\tfrac{2mn}{p} \,\, .
\end{gathered}\end{equation}
Let us stress that, if there exists a smooth solution $\phi: [0,+\infty) \to \bR$ to the Equation \eqref{anseqscal} satisfying the boundary conditions \eqref{icM2}, then it can be extended to a smooth even function on $\bR$.

Notice that, by means of Equation \eqref{anseqscal'}, conditions \eqref{icM2} imply that the solution $u_{a,b,c}$ to the ODE \eqref{ODEu(t)} verifies
\begin{equation} \label{icu(t)}
u_{a,b,c}(1)=0 \,\, , \quad u_{a,b,c}'(1)=\tfrac{4mn}{p} \,\, .
\end{equation}
Therefore, we obtain two values $a(c) ,b(c)$, depending on $c$, by imposing conditions \eqref{icu(t)}: \begin{itemize}
\item[$\bcdot$] if $m=2$, then by Formula \eqref{solm=2} we get
$$
a(c)\=\tfrac25-\tfrac{c}{30}-\tfrac45n \,\, , \quad b(c)\=\tfrac85+\tfrac{c}5+\tfrac45n \,\, ;
$$
\item[$\bcdot$] if $m=3$, then by Formula \eqref{solm=3} we get
$$
a(c)\=\tfrac12-\tfrac{c}{64}-\tfrac{n}2 \,\, , \quad b(c)\=\tfrac32+\tfrac{c}{64}+\tfrac{n}2 \,\, ;
$$
\item[$\bcdot$] if $m>3$, then by Formula \eqref{solm>3} we get
$$\begin{aligned}
a(c) &\= -\frac{(m+1)(4m(2n-p)+4p)+cp}{2p(m+1)(3m-1)} \,\, , \\
b(c) &\= \frac{4m(m-3)(n+p)-cp}{p(3m-1)(m-3)} \,\, .
\end{aligned}$$
\end{itemize}

In all of this three cases it can be directly checked that, for any $c\leq0$, the function $u_c \= u_{a(c),b(c),c}$ is positive and increasing for $t \in (1,+\infty)$. Indeed, by means of conditions \eqref{icu(t)}, there exists $\e_{\zero}>0$ such that $u_c(t)>0$ and $u'_c(t)>0$ for any $t \in (1,1+\e_{\zero})$. Assume by contradiction that there exists $t_{\zero}>1$ such that $u'_c(t)>0$ for any $t \in [1,t_{\zero})$ and $u'_c(t_{\zero})=0$. Then, $t_{\zero}$ is a local maximum point or a stationary point of inflection for $u(t)$, but Equation \eqref{ODEu(t)} implies that
\begin{equation} \label{pignolo}
u_c''(t_{\zero})= 2m(m-1)\tfrac{u_c(t_{\zero})}{t_{\zero}^2} -c +\tfrac{4m(m-1)}{t_{\zero}^2} >0 \,\, ,
\end{equation}
which is not possible. Hence
$$
\f_c: [1,+\infty) \to \bR \,\, , \quad \f_c(r) \= \int_1^r\frac{\diff t}{\sqrt{u_c(t)}}
$$
is smooth, increasing and, by construction, its inverse $\phi_c \= \f_c^{-1}$ solves Equation \eqref{anseqscal} with the initial conditions \eqref{icM2}.
\end{proof}

Constant Chern-scalar curvature metrics on Hirzebruch surfaces have been constructed by Koca and Lejmi by using the B\'erard-Bergery ansatz in \cite[Theorem 1]{koca-lejmi}. Since in complex dimension $m=2$ the complex manifolds $(M_{(4,n)}(\fG,\fK),\bJ)$ reduces to the Hirzebruch surfaces (see Example \ref{ex:hopf}), the next theorem extends their result to $m>2$.

\begin{theorem}\label{thm:Chcsc-M4}
Let $c \in \bR$, $c>0$. Then, all the complex manifolds $(M_{(4,n)}(\fG,\fK),\bJ)$ have a Hermitian, non-K\"ahler, cohomogeneity one metric $\bg$ with $\scal^{\rm Ch}(\bg)=c$.
\end{theorem}

\begin{proof}
Setting $h(0)=1$ and $h(\pi)=k>1$, the smoothness conditions for $f,h$ imply that
\begin{equation} \label{icM4} \begin{gathered}
\phi(0)=1 \,\, , \quad \phi'(0)=0 \,\, , \quad \phi''(0)=\tfrac{2mn}{p} \,\, , \\
\phi(\pi)=k \,\, , \quad \phi'(\pi)=0 \,\, , \quad \phi''(\pi)=-\tfrac{2mn}{kp} \,\, .
\end{gathered}\end{equation}
Let us stress that, if there exists a smooth solution $\phi: [0,\pi] \to \bR$ to Equation \eqref{anseqscal} satisfying the boundary conditions \eqref{icM4}, then it can be extended to a smooth even function on $\bR$ satisfying $\phi(\pi{+}r)=\phi(\pi{-}r)$. Since the case $m=2$ has already been addressed in \cite{koca-lejmi}, we limit ourselves to prove the statement for $m\geq3$.

Assume $m=3$. Then, by means of Equation \eqref{anseqscal'}, conditions \eqref{icM4} imply that the solution $u_{a,b,c}$ given in Formula \eqref{solm=3} verifies
$$
u_{a,b,c}(1)=0 \,\, , \quad u_{a,b,c}(k)=0 \,\, , \quad u_{a,b,c}'(1)=4n  \,\, , \quad u_{a,b,c}'(k)=-\tfrac{4n}{k} \,\, .
$$
By imposing the first three conditions
$$
u_{a,b,c}(1)=0 \,\, , \quad u_{a,b,c}(k)=0 \,\, , \quad u_{a,b,c}'(1)=4n \,\, ,
$$
we obtain three values $a(k),b(k),c(k)$ depending on $k$, that are
$$\begin{aligned}
a(k) &\= -\frac{2k^6(6(n-1)\log(k)k^2+3(k^2-1))}{(8\log(k)-1)k^8+1} \,\, , \\
b(k) &\= \frac{32((n+3)k^8-4k^6-n+1)}{(8\log(k)-1)k^8+1} \,\, , \\
c(k) &\= \frac{32((n+3)k^8-4k^6-n+1)}{(8\log(k)-1)k^8+1} \,\, .
\end{aligned}$$
Set $u_k \= u_{a(k),b(k),c(k)}$ and observe that
$$
u_k'(k)+\tfrac{4n}{k}=\tfrac{\a(k)}{k(8\log(k)-1)k^8+1)} \,\, ,
$$
with
\begin{multline*}
\a(k) \= -4(n+3)k^{10} +32(n+1)k^8\log(k) -4(n-3)k^8 \\
+32(n-1)k^2\log(k) +4(n+3)k^2 +4(n-3) \,\, .
\end{multline*}
Notice that
$$\begin{gathered}
\a(1)=\a'(1)=\a''(1)=0 \,\, , \quad \a'''(1)=512n>0 \,\, , \\
\lim_{k \to +\infty} \a(k) = -\infty
\end{gathered}$$
and so there exists $\tilde{k}>1$ such that $\a(k)>0$ for any $1<k<\tilde{k}$ and $\a(\tilde{k})=0$. Then, we set $u \= u_{\tilde{k}}$, so that the function
$$
\f: [1,\tilde{k}] \to [0,\pi] \,\, , \quad \f(r) \= \int_1^r\frac{\diff t}{\sqrt{u(t)}}
$$
is smooth, increasing and its inverse $\phi \= \f^{-1}$ solves the ODE \eqref{anseqscal} with the boundary conditions \eqref{icM4}.

Assume $m>3$. Then, by means of Equation \eqref{anseqscal'}, conditions \eqref{icM4} imply that the solution $u_{a,b,c}$ given in Formula \eqref{solm=3} verifies
$$
u_{a,b,c}(1)=0 \,\, , \quad u_{a,b,c}(k)=0 \,\, , \quad u_{a,b,c}'(1)=\tfrac{4mn}{p}  \,\, , \quad u_{a,b,c}'(k)=-\tfrac{4mn}{kp} \,\, .
$$
By imposing the first three conditions
$$
u_{a,b,c}(1)=0 \,\, , \quad u_{a,b,c}(k)=0 \,\, , \quad u_{a,b,c}'(1)=\tfrac{4mn}{p} \,\, ,
$$
we obtain three values $a(k),b(k),c(k)$ depending on $k$, that are
$$\begin{aligned}
a(k) &\= \frac{2\big(2(mn-p)k^{m-1}+(mp-2mn-p)k^2-p(m-3)\big)}{p\big(2(m+1)k^{m-1}-(3m-1)k^2+(m-3)k^{-2m}\big)} \,\, , \\
b(k) &\= \frac{4\big(-m(n+p)k^2+(mn-p)k^{-2m}+(m+1)p\big)}{p\big(2(m+1)k^{m-1}-(3m-1)k^2+(m-3)k^{-2m}\big)} \,\, , \\
c(k) &\= \frac{4\tilde{c}(k)}{p(2(m+1)k^{m-1}-(3m-1)k^2+(m-3)k^{-2m})} \,\, ,
\end{aligned}$$
with
\begin{multline*}
\tilde{c}(k) \= 2m(m+1)(m-3)(n+p)k^{m-1}-(3m-1)(m+1)(m-3)p \\
\phantom{\=\;\;} +(m^3(p-2n)+m^2(4n-3p) +m(6mn-p)+3p)k^{-2m} \,\, .
\end{multline*}
Set $u_k \= u_{a(k),b(k),c(k)}$ and observe that
$$
u_k'(k)+\tfrac{4mn}{kp}=\tfrac{\a(k)}{pk^2\b(k)} \,\, ,
$$
where $\a(k), \b(k)$ are the polynomials in $k$ defined by
$$\begin{aligned}
\a(k) &\= -4m(m-3)(n+p)k^{3m+2} +4(m+1)((m-1)p+2mn)k^{3m} -4(3m-1)(mn+p)k^{2m+3} \\
&\qquad\qquad\qquad\quad  +4(3m-1)(mn-p)k^m +4(m+1)((m-1)p-2mn)k^3 -4m(m-3)(p-n)k \,\, , \\
\b(k) &\= 2(m+1)k^{3m-1} -(3m-1)k^{2(m+1)} +(m-3) \,\, .
\end{aligned}$$
Notice that $\b(1)=0$ and 
$$
\b'(k)=2(m+1)(3m-1)(k^{m-3}-1)k^{2m+1} >0 \quad \text{ for any $k>1$} \,\, , 
$$
hence $\b(k)>0$ for any $k>1$. Moreover
$$\begin{gathered}
\a(1)=\a'(1)=\a''(1)=0 \,\, , \\
\a'''(1)=8mn(3m-1)(m-3)(m-1)(m+1)>0 \,\, , \\
\lim_{k \to +\infty} \a(k) = -\infty
\end{gathered}$$
and so there exists $\tilde{k}>1$ such that $\a(k)>0$ for any $1<k<\tilde{k}$ and $\a(\tilde{k})=0$. Then, we set $u \= u_{\tilde{k}}$, so that the function
$$
\f: [1,\tilde{k}] \to [0,\pi] \,\, , \quad \f(r) \= \int_1^r\frac{\diff t}{\sqrt{u(t)}}
$$
is smooth, increasing and its inverse $\phi \= \f^{-1}$ solves the ODE \eqref{anseqscal} with the conditions \eqref{icM4}.

Finally, by means of an argument similar to the one used in the proof of Theorem \ref{thm:chcsc-M2} (see Equation \eqref{pignolo}), it holds necessarily that $c>0$ in all the cases. This concludes the proof. \end{proof}

\appendix

\section{} \label{A}

In this appendix, we provide the details for the proof of Propotision \ref{prop:chern-ricci} computing the Chern-Ricci tensors of B\'erard-Bergery manifolds $(M_{(i,n)}(\fG,\fK),\bJ,\bg)$, of Proposition \ref{propLee} computing their torsion and the Lee form, and of Equation \eqref{ddcw} concerning the pluriclosed condition $\diff\diff^{\,c}\!\w=0$.

\subsection{Proof of Proposition \ref{prop:chern-ricci}} \hfill \par

Let $X,Y \in \gp$ be such that $Q(X,X)=Q(Y,Y)=1$ and $(e_{\a},Je_{\a})$ a $(Q_{\gp},J)$-unitary basis for $\gp$. Set $F \= \tfrac{2mn}{\l p} f$. Then, by using Formulas \eqref{KosChern}, \eqref{defN}, \eqref{JNT}, \eqref{brack*}, \eqref{gJ*}, \eqref{V*J}, \eqref{Jad(T)}, \eqref{lambda}, \eqref{eq:metric}, Proposition \ref{propCh} and Corollary \ref{corLNJ}, we get
\begin{align*}
\bg(&R^{\rm Ch}(\bg)(N,T^*)N,T^*)_{\g(r)} \\
&=\bg(\n_{[N,T^*]}N,T^*)_{\g(r)} -\bg(\n_{N}\n_{T^*}N,T^*)_{\g(r)} +\bg(\n_{T^*}\n_{N}N,T^*)_{\g(r)} \\
&=-\eL_{N}\big(\bg(\n_{T^*}N,T^*)\big)_{\g(r)} +\bg(\n_{T^*}N,\n_{N}T^*)_{\g(r)} +\eL_{T^*}\big(\bg(\n_{N}N,T^*)\big)_{\g(r)} -\bg(\n_{N}N,\n_{T^*}T^*)_{\g(r)} \\
&= -F(r) \tfrac{\p}{\p r}(F'(r)F(r)^2) +2F'(r)^2F(r)^2 \\
&= -F(r)^3F''(r) \,\, , \\[8pt]
\bg(&R^{\rm Ch}(\bg)(X^*,(JX)^*)N,T^*)_{\g(r)} \\
&=-\bg(\n_{[X,JX]^*}N,T^*)_{\g(r)} -\bg(\n_{X^*}\n_{(JX)^*}N,T^*)_{\g(r)} +\bg(\n_{(JX)^*}\n_{X^*}N,T^*)_{\g(r)} \\
&=-\bg(\n_{[X,JX]^*}N,T^*)_{\g(r)} -\eL_{X^*}\big(\bg(\n_{(JX)^*}N,T^*)\big)_{\g(r)} +\bg(\n_{(JX)^*}N,\n_{X^*}T^*)_{\g(r)} \\
&\phantom{=,} +\eL_{(JX)^*}\big(\bg(\n_{X^*}N,T^*)\big)_{\g(r)} -\bg(\n_{X^*}N,\n_{(JX)^*}T^*)_{\g(r)} \\
&=-\l F(r)^2F'(r) +\tfrac12\l^2F(r)^2 +\l F(r)^2F'(r) +\tfrac12\l^2F(r)^2 +\l F(r)^2F'(r) -\tfrac14\l^2F(r)^2\big(2-\tfrac{F(r)^2}{h(r)^2}\big) \\
&\phantom{=,} -\tfrac14\l^2F(r)^2\big(2-\tfrac{F(r)^2}{h(r)^2}\big) \\
&= \l F(r)^2F'(r) +\tfrac12\l^2\tfrac{F(r)^4}{h(r)^2} \,\, , \\[8pt]
\bg(&R^{\rm Ch}(\bg)(N,T^*)X^*,(JX)^*)_{\g(r)} \\
&=\bg(\n_{[N,T^*]}X^*,(JX)^*)_{\g(r)} -\bg(\n_{N}\n_{T^*}X^*,(JX)^*)_{\g(r)} +\bg(\n_{T^*}\n_{N}X^*,(JX)^*)_{\g(r)} \\
&=-\eL_{N}\big(\bg(\n_{T^*}X^*,(JX)^*)\big)_{\g(r)} +\bg(\n_{T^*}X^*,\n_{N}(JX)^*)_{\g(r)} +\eL_{T^*}\big(\bg(\n_{N}X^*,(JX)^*)\big)_{\g(r)} \\
&\phantom{=,} -\bg(\n_{N}X^*,\n_{T^*}(JX)^*)_{\g(r)} \\
&= -F(r)\tfrac{\p}{\p r}(F(r)h(r)h'(r)) +2F(r)^2h'(r)^2 \\
&= -F(r)^2h(r)^2\Big(\tfrac{h''(r)}{h(r)} -\tfrac{h'(r)^2}{h(r)^2} +\tfrac{F'(r)}{F(r)}\tfrac{h'(r)}{h(r)} \Big) \,\, .
\end{align*}

Hence, we obtain
\begin{align*}
\Ric^{\rm Ch[1]}(\bg)&(N,N)_{\g(r)} \\
&= F(r)^{-2}g\big(R^{\rm Ch}(\bg)(N,T^*)N,T^*\big)_{\g(r)} + h(r)^{-2}\sum_{e_{\a} \in \gp}g\big(R^{\rm Ch}(\bg)(N,T^*)e_{\a}^*,(Je_{\a})^*\big)_{\g(r)} \\
&=-F(r)F''(r) -(m{-}1)F(r)^2\Big(\tfrac{h''(r)}{h(r)} -\tfrac{h'(r)^2}{h(r)^2} +\tfrac{F'(r)}{F(r)}\tfrac{h'(r)}{h(r)} \Big) \\
&=F(r)^2\Big(-\tfrac{F''(r)}{F(r)} -(m{-}1)\tfrac{h''(r)}{h(r)} +(m{-}1)\tfrac{h'(r)}{h(r)}\Big(\tfrac{h'(r)}{h(r)} -\tfrac{F'(r)}{F(r)}\Big)\Big) \\
&= \tfrac{2m(m-1)n^2}{p^2}f(r)^2\Big(-\tfrac{f''(r)}{f(r)} +(m{-}1)\Big(-\tfrac{h''(r)}{h(r)} +\tfrac{h'(r)^2}{h(r)^2} -\tfrac{f'(r)}{f(r)}\tfrac{h'(r)}{h(r)}\Big)\Big)
\end{align*}
and
\begin{align*}
\Ric^{\rm Ch[2]}(\bg)&(N,N)_{\g(r)} \\
&= F(r)^{-2}g\big(R^{\rm Ch}(\bg)(N,T^*)N,T^*\big)_{\g(r)} +h(r)^{-2}\sum_{e_{\a} \in \gp}g\big(R^{\rm Ch}(\bg)(e_{\a}^*,(Je_{\a})^*)N,T^*\big)_{\g(r)} \\
&=-F(r)F''(r) +(m{-}1)h(r)^{-2}\big(\l F(r)^2F'(r) +\tfrac12\l^2\tfrac{F(r)^4}{h(r)^2}\big) \\
&=F(r)^2\Big(-\tfrac{F''(r)}{F(r)} +(m{-}1)\l\tfrac{F'(r)}{h(r)^2} +\tfrac12(m{-}1)\l^2 \tfrac{F(r)^2}{h(r)^4}\Big) \\
&= \tfrac{2m(m-1)n^2}{p^2}f(r)^2 \Big(-\tfrac{f''(r)}{f(r)} +\tfrac{2m(m-1)n}{p}\tfrac{f'(r)}{h(r)^2} +\tfrac{2m^2(m-1)n^2}{p^2}\tfrac{f(r)^2}{h(r)^4}\Big) \,\, .
\end{align*}
By the symmetries of the tensors $\Ric^{\rm Ch[i]}(\bg)$, it holds that
$$\Ric^{\rm Ch[i]}(\bg)(T^*,T^*)_{\g(r)} =\Ric^{\rm Ch[i]}(\bg)(N,N)_{\g(r)} \quad \text{ and } \quad \Ric^{\rm Ch[i]}(\bg)(N,T^*)_{\g(r)}=0 \,\, .$$
Moreover, by a direct computation, we get
\begin{align*}
&\bg(R^{\rm Ch}(\bg)(N,(JX)^*)N,T^*)_{\g(r)} = \bg(R^{\rm Ch}(\bg)(N,T^*)N,(JX)^*)_{\g(r)} = 0 \,\, , \\
&\bg(R^{\rm Ch}(\bg)(N,(JX)^*)Y^*,(JY)^*)_{\g(r)} = \bg(R^{\rm Ch}(\bg)(Y^*,(JY)^*)N,(JX)^*)_{\g(r)} = 0
\end{align*}
and hence
$$\Ric^{\rm Ch[i]}(\bg)(N,X^*)_{\g(r)}= \Ric^{\rm Ch[i]}(\bg)(T,X^*)_{\g(r)} = 0 \,\, .$$
Finally, we have
\begin{align*}
\bg(&R^{\rm Ch}(\bg)(N,T^*)X^*,(JX)^*)_{\g(r)} \\
&=\bg(\n_{[N,T^*]}X^*,(JX)^*)_{\g(r)} -\bg(\n_{N}\n_{T^*}X^*,(JX)^*)_{\g(r)} +\bg(\n_{T^*}\n_{N}X^*,(JX)^*)_{\g(r)} \\
&=-\eL_{N}\big(\bg(\n_{T^*}X^*,(JX)^*)\big)_{\g(r)} +\bg(\n_{T^*}X^*,\n_{N}(JX)^*)_{\g(r)} +\eL_{T^*}\big(\bg(\n_{N}X^*,(JX)^*)\big)_{\g(r)} \\
& \phantom{=\;}  -\bg(\n_{N}X^*,\n_{T^*}(JX)^*)_{\g(r)} \\
&= -F(r)\tfrac{\p}{\p r}(F(r)h(r)h'(r)) +2F(r)^2h'(r)^2 \\
&= -F(r)^2h(r)^2\Big(\tfrac{h''(r)}{h(r)} -\tfrac{h'(r)^2}{h(r)^2} +\tfrac{F'(r)}{F(r)}\tfrac{h'(r)}{h(r)} \Big) \,\, , \\[8pt]
\bg(&R^{\rm Ch}(\bg)(X^*,(JX)^*)Y^*,(JY)^*)_{\g(r)} \\
&=-\bg(\n_{[X,JX]^*}Y^*,(JY)^*)_{\g(r)} -\bg(\n_{X^*}\n_{(JX)^*}Y^*,(JY)^*)_{\g(r)} +\bg(\n_{(JX)^*}\n_{X^*}Y^*,(JY)^*)_{\g(r)} \\
&=-\l\bg(\n_{T^*}Y^*,(JY)^*)_{\g(r)} -\eL_{X^*}\big(\bg(\n_{(JX)^*}Y^*,(JY)^*)\big)_{\g(r)} +\bg(\n_{(JX)^*}Y^*,\n_{X^*}(JY)^*)_{\g(r)} \\
& \phantom{=\;} +\eL_{(JX)^*}\big(\bg(\n_{X^*}Y^*,(JY)^*)\big)_{\g(r)} -\bg(\n_{X^*}Y^*,\n_{(JX)^*}(JY)^*)_{\g(r)} \\
&=-\l\bg(\n_{T^*}Y^*,(JY)^*)_{\g(r)} +2\bg(\n_{(JX)^*}Y^*,\n_{X^*}(JY)^*)_{\g(r)} -\eL_{X^*}\big(\bg(\n_{(JX)^*}Y^*,(JY)^*)\big)_{\g(r)}\\
&\phantom{=\;}  +\eL_{(JX)^*}\big(\bg(\n_{X^*}Y^*,(JY)^*)\big)_{\g(r)} \\
&= -\l F(r)h(r)h'(r) +\tfrac12h(r)^2\big(|[X,Y]|_Q^2 +|[X,JY]|_Q^2 +|[JX,Y]|_Q^2 +|[JX,JY]|_Q^2\big) \\
&\phantom{=\;}  +2\l F(r)h(r)h'(r) -\tfrac12F(r)^2\big(Q([T,X],Y)^2+Q([T,X],JY)^2\big) \\
&= h(r)^2 \Big(\l F(r)\tfrac{h'(r)}{h(r)} -\tfrac12\l^2\tfrac{F(r)^2}{h(r)^2}\big(Q(X,Y)^2+Q(JX,Y)^2\big)+ \\
&\phantom{=\;} +\tfrac12\big(|[X,Y]|_Q^2 +|[X,JY]|_Q^2 +|[JX,Y]|_Q^2 +|[JX,JY]|_Q^2\big)\Big) \,\, .
\end{align*} Notice that $$\sum_{e_{\a} \in \gp}\big(Q(X,e_{\a})^2+Q(X,Je_{\a})^2\big) = 1 \,\, .$$
Moreover, setting $(\tilde{e}_{\a})\=(e_{\a},Je_{\a})$, since $[\gp,\gp] \subset \gk$ and $|X|_Q=1$, by the Schur Lemma we obtain
$$
\sum_{e_{\a} \in \gp}\big(|[X,e_{\a}]|_Q^2 +|[X,Je_{\a}]|_Q^2 +|[JX,e_{\a}]|_Q^2 +|[JX,Je_{\a}]|_Q^2\big) = \tfrac1{m-1}\sum_{\tilde{e}_{\a},\tilde{e}_{\b} \in \gp}|[\tilde{e}_{a},\tilde{e}_{\b}]_{\gk}|_Q^2 = 4m \,\, .
$$
Hence, we get 
\begin{align*}
\Ric^{\rm Ch[1]}(\bg)&(X^*,X^*)_{\g(r)} \\
&= F(r)^{-2}g\big(R^{\rm Ch}(\bg)(X^*,(JX)^*)N,T^*\big)_{\g(r)} +h(r)^{-2}\sum_{e_{\a} \in \gp} \!g\big(R^{\rm Ch}(\bg)(X^*,(JX)^*)e_{\a}^*,(Je_{\a})^*\big)_{\g(r)} \\
&=\l F'(r) +\tfrac12\l^2\tfrac{F(r)^2}{h(r)^2} +(m{-}1)\l F(r)\tfrac{h'(r)}{h(r)} -\tfrac12\l^2\tfrac{F(r)^2}{h(r)^2} +2m \\
&=h(r)^2\Big(\l\tfrac{F(r)}{h(r)^2}\Big(\tfrac{F'(r)}{F(r)} +(m{-}1)\tfrac{h'(r)}{h(r)}\Big) +\tfrac{2m}{h(r)^2}\Big) \\
&=h(r)^2\Big(\tfrac{2mn}p\tfrac{f(r)}{h(r)^2}\Big(\tfrac{f'(r)}{f(r)} +(m{-}1)\tfrac{h'(r)}{h(r)}\Big) +\tfrac{2m}{h(r)^2}\Big)
\end{align*}
and
\begin{align*}
\Ric^{\rm Ch[2]}(\bg)&(X^*,X^*)_{\g(r)} \\
&= F(r)^{-2}g\big(R^{\rm Ch}(N,T^*)X^*,(JX)^*\big)_{\g(r)} +h(r)^{-2}\sum_{e_{\a} \in \gp} \!g\big(R^{\rm Ch}(e_{\a}^*,(Je_{\a})^*)X^*,(JX)^*\big)_{\g(r)} \\
&= -h(r)^2\Big(\tfrac{h''(r)}{h(r)} -\tfrac{h'(r)^2}{h(r)^2} +\tfrac{F'(r)}{F(r)}\tfrac{h'(r)}{h(r)} \Big) +(m{-}1)\l F(r)\tfrac{h'(r)}{h(r)} -\tfrac12\l^2\tfrac{F(r)^2}{h(r)^2} +2m \\
&=h(r)^2\Big(-\tfrac{h''(r)}{h(r)} +\tfrac{h'(r)^2}{h(r)^2} -\tfrac{F'(r)}{F(r)}\tfrac{h'(r)}{h(r)} +(m{-}1)\l F(r)\tfrac{h'(r)}{h(r)^3} -\tfrac12\l^2\tfrac{F(r)^2}{h(r)^4} +\tfrac{2m}{h(r)^2}\Big) \\
&= h(r)^2\Big(-\tfrac{h''(r)}{h(r)} +\tfrac{h'(r)^2}{h(r)^2} -\tfrac{f'(r)}{f(r)}\tfrac{h'(r)}{h(r)} +\tfrac{2m(m-1)n}{p}f(r)\tfrac{h'(r)}{h(r)^3} -2\big(\tfrac{mn}{p}\big)^2\tfrac{f(r)^2}{h(r)^4} +\tfrac{2m}{h(r)^2}\Big) \,\, ,
\end{align*}
which concludes the proof of Equations \eqref{ric1}, \eqref{ric2} and \eqref{ric12}.

For what concerns the scalar curvature, we have
\begin{align*}
\scal^{\rm Ch}(\bg)(r) &= 2 F(r)^{-2}\Ric^{\rm Ch[2]}(\bg)(T^*,T^*)_{\g(r)}+2h(r)^{-2}\sum_{e_{\a} \in \gp}\Ric^{\rm Ch[2]}(\bg)(e_{\a}^*,e_{\a}^*)_{\g(r)} \\
&= -2\tfrac{f''(r)}{f(r)} -2(m{-}1)\tfrac{h''(r)}{h(r)} +2(m{-}1)\Big(\tfrac{h'(r)}{h(r)} -\tfrac{f'(r)}{f(r)}\Big)\tfrac{h'(r)}{h(r)} +4m(m{-}1)\tfrac1{h(r)^2} \\
&\phantom{=\;} +\tfrac{4m(m-1)n}{p}\big(f'(r) +(m{-}1)f(r)\tfrac{h'(r)}{h(r)}\big)\tfrac1{h(r)^2} 
\end{align*}
which proves Equation \eqref{scal}. \qed

\subsection{Proof of Proposition \ref{propLee}} \hfill \par

By Equation \eqref{tau} and Corollary \ref{cordw}, the only non-vanishing components of $\t$ are
\begin{align*}
\bg(\t(X^*,T^*),Z^*)_{\g(r)} &= -\tfrac{2mn}{\l p}f(r)\big(h(r)h'(r) +\tfrac{mn}{p}f(r)\big)Q_{\gp}(JX,Z) \,\,, \\
\bg(\t(X^*,N),Z^*)_{\g(r)} &= -\tfrac{2mn}{\l p}f(r)\big(h(r)h'(r) +\tfrac{mn}{p}f(r)\big)Q_{\gp}(X,Z) \,\,,
\end{align*} and so $\t(N,T^*)_{\g(r)}=\t(X^*,Y^*)_{\g(r)}=0$.
Moreover, letting $(e_\a, Je_a)$ be a $(Q_\gp, J)$-unitary basis for $\gp$, we get 
\begin{align*}
\t(N,X^*)_{\g(r)} &= -\big(\tfrac{2mn}{\l p}f(r)\big)^{-2}\Big(\bg(\t(X^*,N),N)_{\g(r)}N_{\g(r)} +\bg(\t(X^*,N),T^*)_{\g(r)}T^*_{\g(r)}\Big) \\
&\phantom{=,} -h(r)^{-2}\sum_{e_{\a} \in \gp}\Big(\bg(\t(X^*,N),e_{\a}^*)_{\g(r)}(e_{\a}^*)_{\g(r)} +\bg(\t(X^*,N),(Je_{\a})^*)_{\g(r)}(Je_{\a})^*_{\g(r)}\Big) \\
&= \tfrac{2mn}{\l p}\tfrac{f(r)}{h(r)^2}\big(h(r)h'(r) +\tfrac{mn}p f(r)\big)X^*_{\g(r)} \,\, .
\end{align*} 
Therefore, by Equations \eqref{JNT} and \eqref{gJ*}, it follows that
$$
\t(T^*,X^*)_{\g(r)} = \tfrac{2mn}{\l p} \tfrac{f(r)}{h(r)^2}\big(h(r)h'(r) +\tfrac{mn}p f(r)\big)(JX)^*_{\g(r)} \,\, .
$$ Finally \begin{align*}
\q(N)_{\g(r)} &= \big(\tfrac{2mn}{\l p}f(r)\big)^{-2}\Big(\bg(\t(N,N),N)_{\g(r)} +\bg(\t(N,T^*),T^*)_{\g(r)}\Big) \\
&\phantom{=,} +h(r)^{-2}\sum_{e_{\a} \in \gp}\Big(\bg(\t(N,e_{\a}^*),e_{\a}^*)_{\g(r)} +\bg(\t(N,(Je_{\a})^*),(Je_{\a})^*)_{\g(r)} \Big) \\
&= \tfrac{4m(m-1)n}{\l p}\tfrac{f(r)}{h(r)^2}\big(h(r)h'(r) +\tfrac{mn}p f(r)\big) \,\, ,
\end{align*}
and analogously one can show that $\q(T^*)_{\g(r)} = \q(X^*)_{\g(r)} = 0$, which concludes the proof. \qed

\subsection{Proof of Equation \eqref{ddcw}} \hfill \par

Let us compute $\diff\diff^{\,c}\!\w$. Since $\bJ \w =\w$, it follows that $$\diff^{\,c}\!\w(A,B,C) = \diff \w(\bJ A,\bJ B,\bJ C)$$ and so \begin{align*}
\diff\diff^{\,c}&\!\w(A,B,C,D) = \\
&= \eL_A\diff\w(\bJ B,\bJ C,\bJ D) -\eL_B\diff\w(\bJ A,\bJ C,\bJ D) +\eL_C\diff\w(\bJ A,\bJ B,\bJ D) -\eL_D\diff\w(\bJ A,\bJ B,\bJ C) \\
&\phantom{=,} -\diff\w(\bJ[A,B],\bJ C,\bJ D) +\diff\w(\bJ[A,C],\bJ B,\bJ D) -\diff\w(\bJ[A,D],\bJ B,\bJ C) \\
&\phantom{=,} -\diff\w(\bJ[B,C],\bJ A,\bJ D) +\diff\w(\bJ[B,D],\bJ A,\bJ C) -\diff\w(\bJ[C,D],\bJ A,\bJ B) \,\, .
\end{align*}
If $A$ is a holomorphic Killing vector field, then $\eL_A\diff\w = \diff\eL_A\w=0$ and so
$$
\eL_A\diff\w(\bJ B,\bJ C,\bJ D) = \diff\w(\bJ [A,B],\bJ C,\bJ D) +\diff\w(\bJ B,\bJ [A,C],\bJ D) +\diff\w(\bJ B,\bJ C,\bJ [A,D]) \,\, .
$$
Therefore, if $A,B,C,D$ are holomorphic Killing, we get 
\begin{align*}
\diff\diff^{\,c}\!\w(A,B,C,D) &= +\diff\w(\bJ [A,B],\bJ C,\bJ D) -\diff\w(\bJ [A,C],\bJ B,\bJ D) +\diff\w(\bJ [A,D],\bJ B,\bJ C) \\
&\phantom{=,} +\diff\w(\bJ [B,C],\bJ A,\bJ D) -\diff\w(\bJ [B,D],\bJ A,\bJ C) +\diff\w(\bJ [C,D],\bJ A,\bJ B) \,\, .
\end{align*}
Letting $\rho(X,Y) = Q_{\gp}(JX,Y)$, we have
\begin{align*}
(\rho{\wedge}\rho)(X,Y,Z,W) =& 2\big(Q(JX,Y)Q(JZ,W) -Q(JX,Z)Q(JY,W) +Q(JX,W)Q(JY,Z)\big)
\end{align*}
and so, by using Proposition \ref{propdw} and Equations \eqref{Jad(T)}, \eqref{eq:metric} we get
\begin{align*}
\diff&\diff^{\,c}\!\w(X^*,Y^*,Z^*,W^*)_{\g(r)} = \\
&= \l Q(JX,Y)\diff\w(Z^*,W^*,N)_{\g(r)} -\l Q(JX,Z)\diff\w(Y^*,W^*,N)_{\g(r)} +\l Q(JX,W)\diff\w(Y^*,Z^*,N)_{\g(r)} \\
& \phantom{=,} +\l Q(JY,Z)\diff\w(X^*,W^*,N)_{\g(r)} -\l Q(JY,W)\diff\w(X^*,Z^*,N)_{\g(r)} +\l Q(JZ,W)\diff\w(X^*Y^*,N)_{\g(r)} \\
&=4\tfrac{mn}{p}f(r) (h(r)h'(r)+ \tfrac{mn}{p}f(r))(\rho{\wedge}\rho)(X,Y,Z,W),
\end{align*}
which concludes the proof. \qed

\end{document}